\documentclass[11pt,dvips]{article}
\usepackage{latexsym}
\usepackage{amsmath,amsthm,amsfonts,amssymb}
\usepackage{enumerate}
\usepackage{graphicx}
\usepackage{pst-all} 
\usepackage[a4paper,textwidth=15cm,textheight=25cm]{geometry} 
\psset{unit=1pt}
\usepackage[margin=1cm]{caption}

\newtheorem{thm}{Theorem}[section]
\newtheorem{dfn}[thm]{Definition} 
\newtheorem{cor}[thm]{Corollary}
\newtheorem{prop}[thm]{Proposition} 
\newtheorem{lem}[thm]{Lemma} 
\newtheorem{conv}[thm]{Convention} 

\theoremstyle{remark}
\newtheorem*{rem*}{Remark}
\newtheorem{rem}[thm]{Remark}
\newtheorem{example}[thm]{Example}

\newcommand{\es}{\emptyset}
\newcommand{\m} {^{-1}} 
\newcommand {\Ra} {\Rightarrow}  

\newcommand {\cala} {{\mathcal {A}}}   
\newcommand {\cale} {{\mathcal {E}}}

\newcommand {\Z} {{\mathbb {Z}}}
\newcommand {\Q} {{\mathbb {Q}}}
\newcommand{\inc}{\subset}
\newcommand{\rk}{\mathop{\mathrm{rk}}}
\newcommand{\ov}{\overline }

\usepackage{pdfsync}

\begin{document}

\title{Quotients and subgroups of Baumslag-Solitar groups}
\author{ Gilbert Levitt}
\date{}

\maketitle

\begin{abstract} We determine all generalized Baumslag-Solitar   groups (finitely generated groups acting on a tree with all stabilizers infinite cyclic)  
which are quotients of a given Baumslag-Solitar group $BS(m,n)$, and (when $BS(m,n)$ is not Hopfian) which of them also admit $BS(m,n)$ as a quotient. We   determine for which values of $r,s$ one may embed  $BS(r,s)$   into a given $BS(m,n)$, and we characterize finitely generated groups which embed into some $BS(n,n)$.
\end{abstract}

\setcounter{tocdepth}{2}

\section{Introduction and statement of results}

This paper
 studies various aspects of the Baumslag-Solitar groups $BS(m,n)=\langle a,t\mid ta^mt\m=t^n\rangle$.
 These groups are HNN extensions of $\Z$, and they are best understood within the class of 
   generalized Baumslag-Solitar (GBS) groups. These are fundamental groups of finite graphs of groups $\Gamma$ with all vertex and edge groups $\Z$
   (equivalently,     finitely generated groups $G$ acting on a tree   with infinite cyclic edge and vertex stabilizers);  they are represented by   $\Gamma$, viewed as a graph labelled by indices of edge groups in vertex groups (see Section \ref{prel}). 
     
It is well-known \cite{ BS, CL} that  Baumslag-Solitar groups are not always Hopfian: they may be isomorphic to proper quotients. In fact, we will see that there exist infinitely many non-isomorphic GBS groups $G$ which are \emph{epi-equivalent} to $BS(4,6)$: there exist  non-injective  epimorphisms $BS(4,6)\twoheadrightarrow  G\twoheadrightarrow  BS(4,6)$.
   
A GBS group $G$ is a quotient of some $BS(m,n)$ if and only if it is 2-generated. The first main result of this paper is the determination of the values of $m,n$ for which a given $G$ is a quotient of   $BS(m,n)$.

It follows from Grushko's theorem (see \cite{Le2}) that the only possibilities for the topology of a labelled graph $\Gamma$ representing a non-cyclic 2-generated GBS group $G$  are: a segment, a circle, or a lollipop (see Figure \ref{disp}; vertices are named $v_i,w_i$, labels are $q_j,r_j,x_j,y_j$; we view a circle as a lollipop with no arc attached). We define numbers $Q,R,X,Y$ as the products of the labels with the corresponding letter, with  
$Q=R=1$ 
if $\Gamma$ is a circle.
 
\begin{figure}[h]
\begin{center}
\begin{pspicture}(100,300)
 
\rput  (0,100)   {

\rput(-160,150){
\psline (0,0)(140,0)
\psline[linestyle=dashed](140,0)(190,0)
\psline (190,0)(270,0)

 \pscircle*( 0,0) {3}
  \pscircle*( 60,0) {3}
   \pscircle*( 120,0) {3}
    \pscircle*( 210,0) {3}
     \pscircle*( 270,0) {3}
     
 \rput(0,-15) {$ v_0$}    
   \rput(60,-15) {$ v_1$}  
     \rput(120,-15) {$ v_2$}  
     \rput(210,-15) {$ v_{k-1}$}  
      \rput(270,-15) {$ v_k$}     
  
   {\blue
   \rput(10,10) {$ q_0$}
   \rput(70,10) {$ q_1$}
    \rput(50,10) {$ r_1$}
      \rput(110,10) {$ r_2$}
        \rput(226,10) {$q_{k-1}$}
          \rput(260,10) {$ r_k$}
   }      

 \rput(400,0)  {segment}
 }
    
   \rput(-160,0){
\psline (0,0)(90,0)
\psline[linestyle=dashed](90,0)(120,0)
\psline (120,0)(180,0)

 \pscircle*( 0,0) {3}
  \pscircle*( 40,0) {3}
   \pscircle*(80,0) {3}
    \pscircle*( 130,0) {3}
     \pscircle*( 180,0) {3}
     
      \pscircle*( 196,47) {3}
        \pscircle*( 196,-47) {3}
          \pscircle*( 232,-75) {3}
           \pscircle*( 232,75) {3}
     
 \rput(0,-15) {$ v_0$}    
   \rput(40,-15) {$ v_1$}  
     \rput(80,-15) {$ v_2$}  
     \rput(130,-15) {$ v_{k-1}$}  
      \rput(195,0) {$ w_0$}     
        \rput(214,44) {$ w_{\ell-1}$}  
         \rput(209,-44) {$ w_{ 1}$}  
         \rput(239,-64) {$ w_{ 2}$}  
           \rput(239,64) {$ w_{ \ell-2}$}  
        
   {\blue
   \rput(10,10) {$ q_0$}
   \rput(48,10) {$ q_1$}
    \rput(35,10) {$ r_1$}
      \rput(74,10) {$ r_2$}
        \rput(143,10) {$q_{k-1}$}
          \rput(170,5) {$ r_k$}
      \rput(193,62) {$  y_{\ell-1}$}  
         \rput(195,-62) {$  x_{ 1}$}  
      \rput(177,40) {$  x_{\ell-1}$}
        \rput(182,-41) {$  y_{ 1}$}     
       \rput(175,17) {$  y_{\ell}$}   
          \rput(173,-13) {$ x_0$}    
      \rput(218,-80) {$  y_{ 2}$}         
    \rput(217,80) {$  x_{ \ell-2}$}                                    
   } 
   \psarc( 260,0) {80}  {95}{265} 
     \psarc[linestyle=dashed]( 260,0) {80}     {265} {95}
   
    \rput(400,0)  {lollipop}
}
}

   \end{pspicture}
\end{center}
\caption{representing 2-generated GBS groups}
\label{disp}
\end{figure}
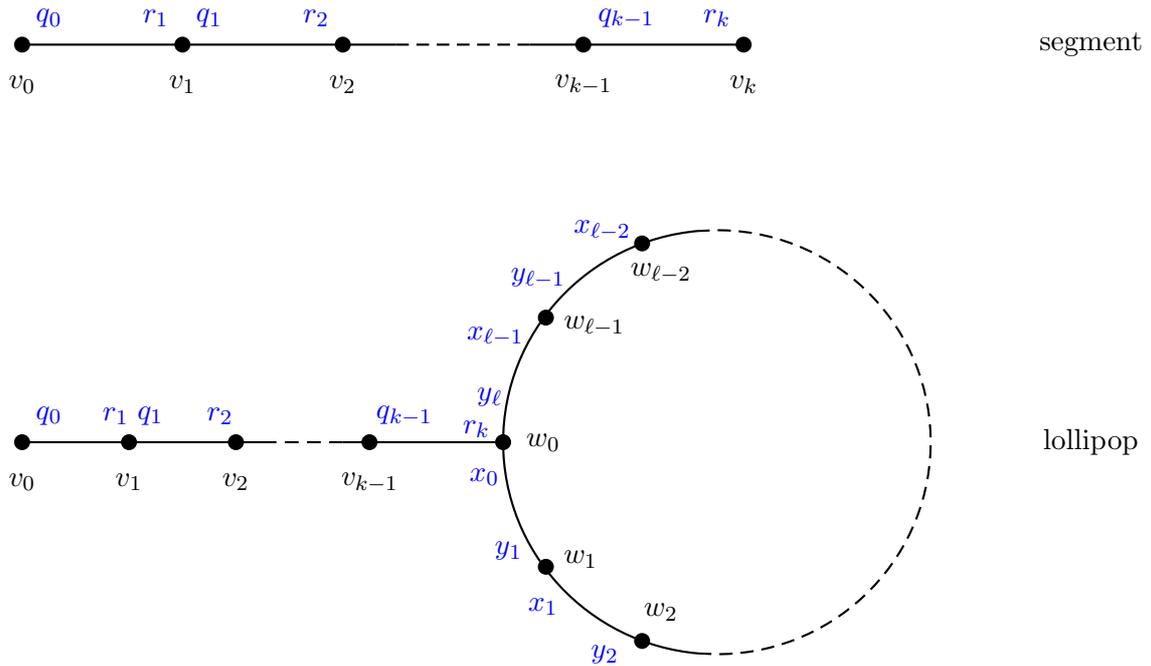

\begin{thm}[Theorem \ref{deBS}] \label{dbs}
Let $\Gamma$ be a labelled graph as on Figure \ref{disp} representing a 2-generated GBS group $G$ other than $\Z$ or the Klein bottle group $K$. 
\begin{itemize}
\item Segment case: $G$ is a quotient of $BS(m,n)$ if and only if $m=n$ and $m$ is divisible by $Q$ or $R$;
\item Circle/lollipop case:  $G$ is a quotient of $BS(m,n)$ if and only if $(m,n)$ is an integral multiple of $(QX,QY)$ or $(QY,QX)$.

\end{itemize}
\end{thm}

Using this, one may describe all GBS quotients of a given $BS(m,n)$. In particular (see Proposition \ref{infq}), one may determine which $BS(m,n)$ have infinitely many GBS quotients  up to isomorphism.

We also consider Baumslag-Solitar quotients of a given 2-generated GBS  group $G$. This is   interesting only for groups represented by a circle or a lollipop (a group represented by a segment has no Baumslag-Solitar quotient, except sometimes the Klein bottle group $K$). By the previous theorem   there is a ``smallest'' Baumslag-Solitar group mapping onto $G$, namely $BS(QX,QY)$, and we determine whether $G$ maps onto $BS(QX,QY)$. The general statement  (Theorem \ref{hop}) is a little technical, so we limit ourselves  here to the groups pictured on  Figure \ref{ptit}.

\begin{figure}[h]
\begin{center}
\begin{pspicture}(100,180)
 
\rput  (-50,80)   {
 
\psline (0,0)(120,0)
 \pscircle ( 180,0) {60} 

 \pscircle*( 0,0) {3}
 
   \pscircle*( 120,0) {3}
   
    {\blue
   \rput(15,10) {$ Q$}
 
    \rput(100,10) {$ R$}
      \rput(117,30) {$Y$}
        \rput(117,-30) {$X$}      
   }      
            Ê}
        \end{pspicture}
\end{center}
\caption{the group $ \langle a,b,t\mid a^Q=b^R, tb^Xt\m =b^Y\rangle$}
\label{ptit}
\end{figure}
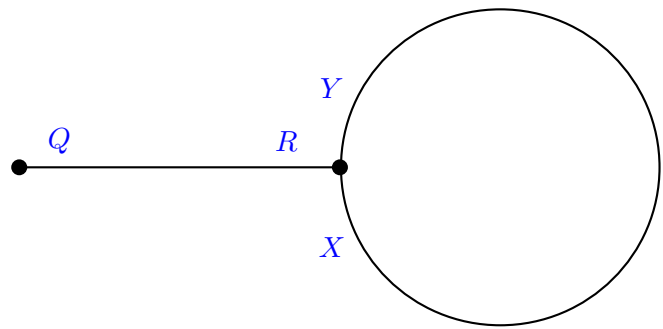

\begin{thm} Suppose that 
  $G=\langle a,b,t\mid a^Q=b^R, tb^Xt\m =b^Y\rangle$ is 2-generated. Then $G$ maps onto $BS(QX,QY)$ if and only if one (or more) of the gcd's $X\wedge QY$, $Y\wedge QX$, $Q\wedge R$   equals 1.
\end{thm}

We say that two groups are \emph{epi-equivalent} if each is isomorphic to a quotient of the other (of course, any group epi-equivalent to a Hopfian group is isomorphic to it). Theorem \ref{hop} may be viewed as the determination of the GBS groups which are epi-equivalent to a Baumslag-Solitar group (see Corollary \ref{epi}).    

 Recall \cite{CL} that $BS(m,n)$ is Hopfian if and only if $m$ or $n$ equals $\pm1$, or $m$, $n$ have the same set of prime divisors (Remark \ref{nh} contains a  simple geometric argument showing non-Hopficity, illustrated on Figure \ref{hopf}).
If $m,n$ are integers which are not prime and not coprime, and $BS(m,n)$ is not Hopfian, then (Corollary \ref{iee}) there are infinitely many GBS groups which are epi-equivalent to $BS(m,n)$. 

Theorem \ref{hop}, together with a result from \cite{Le} or \cite{Me}, also yields the following characterization of GBS groups $G$ which are large (some finite index subgroup maps onto a non-abelian free group):  \emph{$G$ is large unless it is a quotient of some $BS(m,n)$ with $m,n$ coprime} (see Corollary \ref{pointi}).

In Section \ref{bss} we use immersions of graphs of groups \cite{Ba} (see also \cite{Dud}) to find Baumslag-Solitar  subgroups in GBS groups. In particular, we show that, if $\frac mn $ is a non-zero rational number written in lowest terms with $m\ne\pm n$, and $G$ contains elements $a$ and $t$ satisfying 
$ta^{qm}
t\m=a^{qn}$ for some $q\ge1$, then $G$ contains $BS(m,n)$. In other words, one may define the set of moduli of $G$ (see Section \ref{prel}) in terms of its Baumslag-Solitar subgroups. 

We deduce that
\emph  {a   GBS group is   residually finite if and only if it is solvable or virtually $F \times\Z$ with $F$ free} (see \cite{Mes} for   Baumslag-Solitar groups). This is   well-known (see \cite {Wi} page 128), but we have not found a reference in the literature, so we provide a proof. Our argument shows that a GBS group is residually finite if and only if all its Baumslag-Solitar subgroups are.

We also determine under which conditions one may embed a Baumslag-Solitar group into another. 

\begin{thm} [Theorem \ref{obstr}]\label{obstr0}
Assume that  
$(r,s)\ne(\pm1,\pm1)$. 
Then  $BS(r,s)$ embeds into $BS(m,n)$ if and only if the following hold:  
\begin{enumerate}
\item $\frac rs$ is a power of $\frac mn$;
\item if $m$ and $n$ are divisible by $p^\alpha$ but not by $p^{\alpha+1}$, with $p$ prime and $\alpha\ge0$, then $r$ and $s$ are not divisible by $p^{\alpha+1}$; in particular, any prime dividing $rs$ divides $mn$;
\item if $m$ or $n$ equals $\pm1$, so does $r$ or $s$.
\end{enumerate}
\end{thm}

The second condition says for instance that $BS(12,20)$ does not embed into $BS(6,10)$.
It follows from the theorem that $BS(m,n)$ contains infinitely many non-isomorphic Baumslag-Solitar groups if and only if $m\ne\pm n$. For instance, $BS(2,3)$ contains $BS(2^{a+b}3^c, 2^b3^{a+c})$ for all $a,b,c\ge0$, while $BS(3,3)$ only contains $\Z^2=BS(1,1)$.

At the end of the paper we characterize subgroups of $BS(n,n)$ and $BS(n,-n)$. In particular:

\begin{thm}[Corollary \ref{mmm}] \label{mmmm}
A finitely generated  group $G$ embeds into some $BS(n,n)$ if and only if $G$ is  torsion-free, and either 
$G$ is free or $G $ is a central   extension of $\Z$ by a free product of cyclic groups. 
\end{thm}

The paper is organized as follows.  

Section \ref{prel}  consists of preliminaries about GBS groups. We also use results from \cite{Le2} in order to describe 2-generated GBS groups.
Section \ref{gene} consists of generalities about maps between GBS groups. In particular, 
we show that (roughly speaking) a GBS group $G$ maps onto a Baumslag-Solitar group if and only if it is represented by labelled graphs with first Betti number at least 1  
(see Lemma \ref{qbs} for a correct statement), and we illustrate how strongly the descending chain condition fails for GBS groups.

  In Section \ref{sens1} we prove Theorem \ref{dbs}: given a 2-generated GBS group $G$, we find the values of $m$ and $n$ for which there is an epimorphism $BS(m,n)\twoheadrightarrow G$. Corollaries are derived in Section \ref{cons}; in particular, we  determine how many non-isomorphic GBS quotients a given  $BS(m,n)$ has.
  
  In Section \ref{sens2} we state and prove Theorem \ref{hop}, determining whether a given GBS group $G$ is epi-equivalent to a Baumslag-Solitar group (see Corollary \ref{epi}). We also find which $BS(m,n)$ are epi-equivalent to infinitely many non-isomorphic GBS groups. 
  
  Section \ref{bss} is devoted to the construction of Baumslag-Solitar subgroups. In particular, we prove Theorem \ref{obstr0}. In the last section we study subgroups of $BS(n,n)$.
  
 This work was partially supported by 
ANR-07-BLAN-0141-01 and ANR-2010-BLAN-116-03.  

 \tableofcontents

\section{Preliminaries} \label{prel}

\subsection{Generalities on GBS groups}

We refer to \cite{FoJSJ,FoGBS,Le} for basic facts about GBS groups.

GBS groups are represented by labelled graphs. 
A \emph{labelled graph} is a finite   graph $\Gamma$ where each oriented edge $e$ has a label $\lambda_e$, a nonzero integer (possibly negative). When drawing pictures, we place $\lambda_e$ along $e$ near its origin (in blue).

We denote by $V$ the set of vertices of $\Gamma$, and by $\cale$ the set of non-oriented edges. We view a non-oriented edge as $\varepsilon=(e,\tilde e)$, where $\tilde e$ is $e$ with the opposite orientation. We denote by $v=o(e)$ the origin of $e$, and by $E_v$ the set of oriented edges with origin $v$. The cardinality $ | E_v | $ of $E_v$ is the \emph{valence} of $v$. A vertex is \emph{terminal} if it has valence one. 

We say that $\lambda_e$ is the \emph{label of $e$ near the vertex $o(e)$}, and that $\lambda_e$ is a label \emph{carried by $e$ or $\varepsilon$}. There are $ | E_v | $ labels near a vertex $v$. 

We write $m\wedge n$ for the greatest common divisor (gcd) of two integers $m,n$ (i.e.\ the positive generator of $\langle m\Z,n\Z\rangle$); by   convention, $m\wedge n>0$ regardless of the sign of $m,n$. On the other hand, $lcm(m,n):=\frac{mn}{m\wedge n}$ may be negative.
 
A connected labelled graph defines a graph of groups. All edge and vertex groups are $\Z$, and the inclusion from the edge group $G_e$ to the vertex group $G_{o(e)}$ is multiplication by $\lambda_e$. The fundamental group $G$ of the graph of groups  is the \emph{GBS group represented by $\Gamma$} (we do not always assume that $\Gamma$ is connected, but we implicitly do   whenever we refer to the group it represents). The group $G$  may be presented as follows (see an example on Figure \ref{ptit}). 

Choose a maximal subtree $\Gamma_0\inc \Gamma$. There is one generator $a_v$ for each vertex $v\in V$, and one generator $t_\varepsilon$ (stable letter) for each $\varepsilon$ in $\cale_0$, the set of 
non-oriented edges  
not in $\Gamma_0$. Each non-oriented edge $\varepsilon=(e,\tilde e)$ of $\Gamma$ contributes a relation $R_\varepsilon$. The relation is $(a_{o(e)})^{\lambda_e} = (a_{o(\tilde e)})^{\lambda_{\tilde e}}$ if $\varepsilon$ is   in $\Gamma_0$, and $t_\varepsilon (a_{o(e)})^{\lambda_e}t_\varepsilon\m = (a_{o(\tilde e)})^{\lambda_{\tilde e}}$ if $\varepsilon$ is  not  in $\Gamma_0$ (exchanging $e$ and $\tilde e$ amounts to replacing $t_\varepsilon$ by its inverse). This will be called a \emph{standard presentation} of $G$, and the generating set  will be called a \emph{standard generating set} (associated to $\Gamma$ and $\Gamma_0$).

The group $G$ represented by $\Gamma$ does not change if one changes the sign of all labels near a given vertex $v$, or if one changes the sign of the labels $\lambda_e$, $\lambda_{\tilde e}$ carried by a given non-oriented edge $\varepsilon$. These will be called \emph{admissible sign changes}. 

A GBS group is \emph{elementary} if it is isomorphic to $\Z$, or $\Z^2$, or the Klein bottle group $K=\langle x,y\mid x^2=y^2\rangle = \langle a,t\mid tat\m=a\m\rangle$. These are the only virtually abelian GBS groups, and they have very special properties. Unless mentioned otherwise, our results apply to all GBS groups, but 
 we do not always provide proofs for elementary groups. 

A   labelled graph $\Gamma$ is \emph{minimal} if its Bass-Serre tree  $T$  contains no proper $G$-invariant subtree; this is equivalent to no label near a terminal vertex being equal  to $\pm1$. If $\Gamma$ is minimal, then $G$ is elementary if and only if $T$ is a point or a line. 

The graph $\Gamma$ is \emph{reduced} \cite{Fodef} if any edge $e$ such that $\lambda_e=\pm1$ is a loop ($e$ and $\tilde e$ have the same origin). Any labelled graph may be made reduced by a sequence of  \emph{elementary collapses} (see Figure \ref{col}); these collapses do not change $G$. 

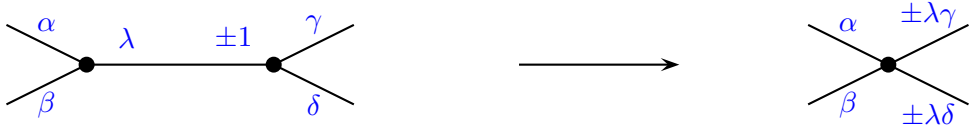
\begin{figure}[h]
\begin{center}
\begin{pspicture}(100,105) 
\rput   (0,50){
\rput(-100,0)
{
   \psline  (0,0)(70,0)   
    \psline  (0,0)(-30,15)
     \psline  (0,0)(-30,-15)
  \psline(70,0)(100,15)
    \psline(70,0)(100,-15)

 \pscircle*( 0,0) {3}
   \pscircle*( 70,0) {3} 
   
      {\blue
   \rput(-15,15) {$ \alpha$}
     \rput(-15,-15) {$ \beta$}
 
   \rput(85,15) {$ \gamma$}
     \rput(85,-15) {$\delta$}
 
    \rput(15,10) {$ \lambda$}
      \rput(55,10) {$ \pm1$}
             
   }      

 }
\rput(200,0)
{

 \psline  (0,0)(-30,15)
 \psline  (0,0)(-30,-15)
 \psline  (0,0)(30,15)
 \psline  (0,0)(30,-15)

 \pscircle*( 0,0) {3}
 
    {\blue
     \rput(-15,15) {$ \alpha$}
     \rput(-15,-15) {$ \beta$}
 
   \rput(15,19) {$\pm\lambda \gamma$}
     \rput(15,-19) {$\pm\lambda\delta$}
   }      

}   
   \rput (60,0){ 
\psline[arrowsize=5pt]  {->}(0,0)(60,0)
} 
 }
\end{pspicture}
\end{center}
\caption{elementary collapse ($G$ does not change)
} \label{col}
\end{figure}

There may be infinitely many reduced labelled graphs representing a given $G$.
 One case when uniqueness holds (up to admissible sign changes) is when $\Gamma$ is \emph{strongly slide-free}  \cite{Fodef}: if $e$ and $e'$ are edges with the same origin, $\lambda_e$ does not divide $\lambda_{e'}$.

 In general,  two reduced labelled graphs $\Gamma_1,\Gamma_2$  represent isomorphic GBS groups $G_1$, $G_2$ if and only if one can pass from $\Gamma_1$ to $\Gamma_2$ by  a finite sequence of admissible sign changes, slide moves, induction moves, $\cala^{\pm1}$-moves \cite{CFWh} (there is one exception to this fact, when $G_1=G_2=K$; the two presentations given above correspond to different reduced graphs: the first is an edge with labels 2 and 2, the second a loop with labels 1 and $-1$).  

We denote by $\beta(\Gamma)$ the first Betti number of $\Gamma$. If $G\ne K$, all labelled graphs representing 
$G$ have the same $\beta$, and we sometimes denote it by $\beta(G)$.
 
Let $G$ be represented by $\Gamma$. 
An element $g\in G$ is \emph{elliptic} if it fixes a point in the Bass-Serre tree $T$ associated to $\Gamma$, or equivalently if some conjugate of $g$ belongs to a vertex group of $\Gamma$. All elliptic elements are pairwise  \emph{commensurable} (they have a common power). An element which is not elliptic is \emph{hyperbolic}. If $G$ is non-elementary, ellipticity or hyperbolicity does not depend on the choice of $\Gamma$ representing $G$.  

 Any finitely generated subgroup of $G$ is free (if it acts freely on $T$) or is a GBS group (if the action is not free).
 
The quotient of $G$ by the subgroup generated by all elliptic elements may be identified with the topological fundamental group $\pi_1^{top}(\Gamma)$ of the graph $\Gamma$, a free group of rank $\beta(\Gamma)$.    
In particular, any generating set of $G$ contains at least $\beta(\Gamma)$ hyperbolic elements. 

If $G$ is non-elementary, there is a \emph{modular homomorphism} $\Delta_G:G\to\Q^*$ associated to $G$. One may compute $\Delta_G(g)$ as follows: given any non-trivial elliptic element $a$, there is a non-trivial relation $ga^m g\m=a^n$,
and
$\Delta_G(g)=\frac mn$ (given $g$, the numbers $m,n$  may depend on the choice of $a$, but $m/n$ does not).  In particular, if $G=BS(m,n)=\langle a,t\mid ta^mt\m=t^n\rangle$, then $\Delta_G$ sends $a$ to 1 and $t$ to $m/n$.
Note that $h$ has to be elliptic is there is a relation $gh^m g\m=h^n$ with $m\ne\pm n$ since its translation length in $T$ satisfies $ |m |\ell(h)= |n | \ell(h)$. 

Elliptic elements have modulus 1, so $\Delta_G$ factors through $\pi_1^{top}(\Gamma)$ for any $\Gamma$ representing $G$. One may define $\Delta_G$ directly in terms of loops in $\Gamma$: if $\gamma\in\pi_{1}^{top}(\Gamma)$ is represented by an edge-loop $(e_1,\dots,e_m)$, then $\Delta_G(\gamma)=\prod_{i=1}^m\frac{\lambda_{e_i}}{\lambda_{\tilde e_i}}$. Note that $\Delta_G$ is trivial if $\Gamma$ is a tree. 

$\Delta_G$ is trivial if and only if the center of $G$ is non-trivial. In this case the center is   cyclic and only contains  elliptic elements (see \cite{Le}); moreover, there is an epimorphism $G\twoheadrightarrow\Z$ whose kernel contains no non-trivial elliptic element (see Proposition 3.3 of \cite{Le}). 

A non-elementary $G$ is \emph{unimodular} if the image of $\Delta_G$ is contained in $\{1,-1\}$. This is equivalent to $G$ having a normal infinite cyclic subgroup, and also to $G$ being virtually $F_n\times\Z$ for some $n\ge2$ (see \cite{Le}). 

\subsection{2-generated GBS groups} \label{2g}

Any quotient of a Baumslag-Solitar group is 2-generated. Conversely, we note:

\begin{lem}\label{qdeBS}
Any 2-generated GBS group $G$ is a quotient of some Baumslag-Solitar group $BS(m,n)$.
\end{lem}   

\begin{proof}
We may assume that $G$ is non-elementary.
There exists a generating pair $(a,t)$ with $a$ elliptic (this may be deduced from   \cite{Le2}, or from the  general fact 
that any finite generating set of a non-free group acting on  a tree is Nielsen-equivalent to one containing an elliptic element). 
Since $a$ and $tat\m$ are commensurable, $a$ and $t$ satisfy a non-trivial relation $ta^mt\m=a^n$ and $G$ is a quotient of $BS(m,n)$.
\end{proof}

In Sections \ref{sens1}, \ref{cons}, \ref{sens2} we shall study GBS groups $G$ which are    2-generated, 
and not cyclic. Note that the isomorphism problem is solvable for these groups by \cite{CFIso}.

We always  represent $G$ by a reduced labelled graph $\Gamma$. 
2-generation implies that the sum of the first Betti number $\beta (\Gamma)$ and the number of terminal vertices $t(\Gamma)$   is at most 2. This is a consequence of    Theorem 1.1 of \cite{Le2}, but it   may be deduced directly from Grushko's theorem  as follows. Add to  a   standard presentation of $G$   the relations $a_v=1$ if $v$ is not terminal, and $(a_v)^{p_v}=1$ if $v$ is terminal and $p_v$ is a prime dividing the label near $v$. This maps $G$ onto the free product of $\beta_1(\Gamma)+t(\Gamma)$ non-trivial cyclic groups, so $\beta_1(\Gamma)+t(\Gamma)\le2$. 

We deduce that  there are only three possibilities for the topological type of $\Gamma$: it may be homeomorphic to a segment, a circle, or a lollipop. 

We picture the corresponding graphs on Figure \ref{disp}, often viewing a circle as a special case of a lollipop.  The number of edges is $k$    or $k+\ell$ respectively.
 Since   $\Gamma$ is reduced,   no label is equal to $\pm1$, except possibly $x_0$ and $y_1$ when $\ell=1$. We denote by $v_i$  the vertex next to the labels $q_i$ and/or $r_i$. In the  lollipop case,  the vertex next to the labels $x_j$ and/or $y_j$ is denoted by $w_j$; there is a special vertex $v_k=w_0=w_\ell$, we will usually denote it by $w_0$.

In the lollipop case, $G$ has a standard presentation with generators $a_0,\dots, a_i,\dots,a_{k-1}$ (associated to the vertices $v_i$), $ b_0,\dots,b_j,\dots, b_{\ell-1}$ (associated to the $w_j$'s), $\tau$ (stable letter) and relations
\begin{eqnarray} 
(a_i)^{q_i}&=&(a_{i+1})^{r_{i+1}}    \textrm{ \quad for $i=0,\dots,k-2$   }     \\
(a_{k-1})^{q_{k-1}}&=&(b_0)^{r_k} \\
(b_j)^{x_j}&=&(b_{j+1})^{y_{j+1}} \textrm{ \quad for $ j=0,\dots,\ell-2$  }  \\
\tau(b_{\ell-1})^{x_{\ell-1}}\tau\m&=&(b_0)^{y_\ell}. 
\end{eqnarray}

In the segment case, we delete $b_1,\dots, b_{\ell-1}, \tau$ and relations (3) and (4), and we write $a_k$ instead of $b_0$.
 In the circle case, we delete $a_0,\dots, a_{k-1}$ and relations (1) and (2).

\begin{dfn}[Q,R,X,Y] \label{qrxy}
Given $\Gamma$ as above, we define numbers as follows: 
\begin{itemize}
 \item if $\Gamma$ is a segment, we let $Q=\prod_{i=0}^{k-1} q_i$ and $R=\prod_{i=1}^{k} r_i$;
 \item   if $\Gamma$ is  a lollipop, we let $Q=\prod_{i=0}^{k-1} q_i$, $R=\prod_{i=1}^{k} r_i$, $X=\prod_{j=0}^{\ell-1} x_j $, $Y=\prod_{j=1}^{\ell } y_j$ (with $Q=R=1$ if $\Gamma$ is a circle).
\end{itemize}
\end{dfn}

We now   recall the definition of a \emph{plateau} \cite{Le2}. For $p$ a prime, a $p$-plateau is a nonempty connected subgraph   $P\inc \Gamma$ such that, if $e$ is an oriented edge whose origin $v$ belongs to $P$, then the label $\lambda_e$ of $e$ near $v$ is divisible by $p$ if and only if $e$ is not contained in $P$. A plateau is a subgraph which is a $p$-plateau for some $p$. In particular, every terminal vertex is a plateau, and $\Gamma$ is a plateau. A plateau is \emph{interior} if it contains no terminal vertex. 

Theorem 1.1 of \cite{Le2} states that the rank $\rk(G)$ of $G$ (minimal cardinality of a generating set)  is equal to $\beta(\Gamma)+\mu(\Gamma)$, with $\beta(\Gamma)$ the first Betti number and $\mu(\Gamma)$ the minimal cardinality of a set of vertices meeting every plateau. Thus:

\begin{lem} \label{gen} Let $\Gamma$ be a reduced labelled graph representing a 2-generated GBS group.
 \begin{itemize}
  \item If $\Gamma$ is not a circle, there is no interior plateau.
 \item  If $\Gamma$ is   a circle, there is a vertex belonging to every plateau.
 \qed
\end{itemize}
\end{lem}

\begin{conv}\label{bc}
 When $\Gamma$ is a circle, we always choose the numbering of vertices   so that $w_0$ belongs to every plateau. 
\end{conv}

 With this convention, Lemma \ref{gen} implies:

\begin{lem} \label{copr}
 \begin{itemize}
 \item  $q_j$ and $r_i$    are coprime for $1\le i\le j\le k-1$; 
 \item If $\Gamma$ is a lollipop (possibly a circle), then   $x_j$ and $y_i$ are coprime for $1\le i\le j\le \ell -1$;
if a prime $p$ divides $R$, then $p$ divides $X$ or $Y$, but not both.  
\end{itemize}

\end{lem}

\begin{proof}
We argue by contradiction, showing the existence of a    plateau   contradicting Lemma \ref{gen}.  Suppose that a prime $p$ divides $q_j$ and $r_i$, with   $i\le j$ and $j-i$ minimal. Then the segment bounded by the vertices $v_i$ and $v_j$ is an interior  $p$-plateau, a contradiction. In the lollipop case, the same argument applies to $x_j$ and $y_i$.

Now suppose that $p$ divides $r_i$, with $i $ maximal. If $p$ does not divide $XY$, then $v_i$ is a boundary point of a $p$-plateau (containing the circle). If $p$ divides both $X$ and $Y$, say $p | x_j$ with $j\ge0$ minimal and $p | y_{j'}$ with $j'\le\ell$ maximal (and necessarily $j<j'$), then $v_i$, $w_j$ and $w_{j'}$ bound a $p$-plateau.
\end{proof}

Consider a standard presentation of a GBS group $G$, with generators $a_v$ and $t_\varepsilon$.
 Proposition 3.9 of  \cite{Le2} states  that, given any subset $V_1\inc V$ meeting every plateau, $G$ is generated by the stable letters $t_\varepsilon$ and the $a_v$'s for $v\in V_1$ (in other words,   the $a_v$'s for $v\notin V_1$ may be removed from the generating set).

 Applied to a 2-generated $G$, this says that 
$G$ 
  is generated by $a_0$ and $a_k$ in the segment case, by $a_0$ and $\tau$  in the lollipop case. 
      If $\Gamma$ is a circle, 
the generators $a_i$ do not exist; 
$G$ is generated by $b_0$ and $\tau$ (assuming Convention \ref{bc}). 
 
\begin{lem} \label{lesp}
Let $(s,s')$ be a generating pair of $G$, with $s$ elliptic. Then $s$ is conjugate to $(a_v)^{r_0}$, with $r_0\in\Z$ and:
 \begin{itemize}
 \item if $\Gamma$ is a segment,      $v$ is one of the  terminal vertices $v_0$ or $v_k$ of $\Gamma$;
  \item if $\Gamma$ is a circle,   
  $v $ belongs to every plateau;
 \item if $\Gamma$ is a lollipop,     
 $v=v_0$ is  the terminal vertex of $\Gamma$.  
   \end{itemize}
 Moreover, ${r_0}\wedge p=1$ for each $p$ such that $v$ belongs to  a proper $p$-plateau. 
 \end{lem}
 
 \begin{rem} \label{rr}
 We may be more explicit about ${r_0}$, using the notations of Definition \ref{qrxy}. If $\Gamma$ is a segment, then ${r_0}\wedge Q=1$  if  $s$ is (up to conjugacy) a power of $a_0$, and ${r_0}\wedge R=1$  if  $s$ is (up to conjugacy) a power of $a_k$. In the circle/lollipop case, 
  ${r_0}$ is coprime with $QX\wedge QY$  because  $v$ belongs to a $p$-plateau $P$ with $w_0\notin P$ if and only if $p$ divides $Q$, to a proper $p$-plateau containing $w_0$ if and only if   $p$ divides $X$ and $Y$ but not $Q$.
  \end{rem}
 
\begin{proof}  If $\Gamma$ is a segment of length $>1$, 
 adding to a standard presentation the  relations $a_v=1$ for $v $ not terminal yields a group of rank 2 (the free product of two nontrivial cyclic groups). It follows that $s$ must be conjugate to   $(a_v)^{r_0}$ with $v$ terminal. If $v$ belongs to a proper $p$-plateau $P$, we add the relations $(a_w)^p=1$ for $w\in P$, and $a_w=1$ if $w\notin P$ and $w$ is not terminal. We obtain a group of rank 2 by Lemma 3.4 of \cite{Le2}, so $p$ does not divide ${r_0}$.
 
In the circle/lollipop case, $s$ is conjugate to $(a_v)^{r_0}$ with $v$ belonging to every plateau  by  Corollary 3.6 of \cite{Le2} (which is proved also  by adding relations killing powers of elliptic elements). Moreover, ${r_0}$ is not divisible by $p$ if $v$ belongs to a $p$-plateau by Corollary 3.7 of \cite{Le2}. 
\end{proof}

We conclude this subsection with the following useful fact.

\begin{lem}\label{padun}
If a GBS group $G$ is 2-generated and is not a solvable $BS(1,n)$, it may be represented by a labelled graph $\Gamma$ with no label equal to $\pm1$.
\end{lem}

\begin{proof}
Represent $G$ by a reduced $\Gamma$ with minimal number of edges. If some label equals $\pm1$, then $\Gamma$ is a lollipop with $\ell=1$: the circle subgraph consists of a single edge with at least one label equal to $\pm1$, say $y_1=\pm1$. Since $G$ is not solvable, $\Gamma$ is not a circle, so there is another edge $e$ attached to the vertex $w_0$, with label $r_k$. By Lemma \ref{copr}, every prime dividing $r_k$ divides $x_0$. One may then perform induction moves (see \cite{FoGBS, CFIso})   to make $r_k$ equal to $\pm 1$. Collapsing the edge $e$ (as on Figure \ref{col}) yields a graph with one less edge representing $G$, a contradiction.\end{proof}
  
\section{Maps between GBS groups} \label{gene}

\begin{lem}\label{koc}
 Let $f:G\to  G'$ be a homomorphism between   GBS groups, with $f(G)$ non-elementary.  If $a$ is elliptic in $G$, then $f(a)$ is elliptic in $G'$.
 \end{lem}

The   result of \cite{Ko} is a special case.

\begin{proof} (compare \cite{Fodef}, proof of Corollary 6.10).
Let $T$ be  a tree on which $G'$ acts with cyclic stabilizers. We let $G$ act on $T$ through $f$. By way of contradiction, suppose that $a$ acts hyperbolically on $T$. It has an axis $A$, the unique $a$-invariant line; it is also the axis of any power $a^n$ with $n\ne 0$. If $g\in G$, the elements $a$ and $gag\m$ have a common power, so $g$ leaves $A$ invariant. It follows that   the line $A$ is $f(G)$-invariant, so $f(G)$ is elementary 
  ($\Z$, $\Z^2$, or $K$),  
a contradiction. 
\end{proof}

\subsection{Elliptic-friendly homomorphisms}

\begin{lem}\label{dic}
 Let $f:G\to G'$ be a homomorphism from a non-elementary  GBS group to a torsion-free group. Suppose that  
  $\ker f$ contains a non-trivial elliptic element. Then $\ker f$ contains all elliptic elements. If $\Gamma$ is any labelled graph representing $G$,  then $f$ factors through the topological fundamental group of $\Gamma$, and therefore   $\beta (\Gamma)\ge \rk(f(G))$.
\end{lem}

\begin{proof} 
The first assertion holds because any two elliptic elements of $G$ are commensurable (they have a common power). If $\Gamma$ is any graph of groups, its topological fundamental group   is the quotient of $G$ by the (normal) subgroup generated by all elliptic elements. 
\end{proof}

We say that $f$ is \emph{elliptic-friendly} if $\ker f$ contains no non-trivial elliptic element. Lemma \ref{dic} clearly implies:

\begin{cor}\label{ef}  
If $G$ is a non-elementary GBS group with $\beta (G)=1$, and $G'$ is   torsion-free, any homomorphism 
$f:G\to G'$ with non-cyclic image 
is elliptic-friendly. \qed
\end{cor}

\begin{lem} \label{modeg}
 Let $f:G\to  G'$ be an elliptic-friendly homomorphism between GBS groups, with $G$ non-elementary.
 
 If $f(G)$ is non-elementary, then    
 $\Delta_G=\Delta_{G'}\circ f$; in particular, $\Delta_G$ and $\Delta_{G'}$ have the same image in $\Q^*$ if $f$ is onto.

If $f(G)$ is elementary, $\Delta_G$ is trivial when $f(G)$ is $\Z$ or $\Z^2$, has image contained in $\{1,-1\}$ when $f(G)$ is the Klein bottle group $K$.
\end{lem}

See Section \ref{prel} for the definition of the modular homomorphism $\Delta_G:G\to\Q^*$.

\begin{proof} 

If $\Delta_G(g)=\frac mn$, there is a non-trivial elliptic element $a\in G$ such that $ga^m g\m=a^n$. One deduces $f(g)f(a)^m f(g)\m=f(a)^n$, with $f(a)$   non-trivial. 

This equation implies $\Delta_{G'}(f(g))=\frac mn$
if $f(G)$ is 
  non-elementary,
because   $f(a)$ is   elliptic by Lemma \ref{koc}.   
It  implies $m=n$ in $\Z$ and $\Z^2$, $m=\pm n$ in $K$. 
\end{proof}

\subsection{Epimorphisms between Baumslag-Solitar groups}
 
\begin{lem} \label {eBS} 
 
 There is an epimorphism $f:BS(m,n)\twoheadrightarrow BS(m',n')$ if and only if one of the following holds:
\begin{itemize}
 \item  $(m,n)$ is an integral multiple of $(m',n')$ or $(n',m')$;
 \item $BS(m',n')=BS(1,-1) $ is the Klein bottle group $K$, and $m=n$ with $m$ even. 
\end{itemize}
\end{lem}

This is Proposition A.11 of \cite{So} or Theorem 5.2 of  \cite{DRT} in the non-elementary case.

\begin{proof}
Write $G=BS(m,n)=\langle a,t\mid ta^m t\m=a^n\rangle$ and $G'=BS(m',n')=\langle a',t'\mid t'a'{}^{m'} t'{}\m=a'{}^{n'}\rangle$.

The ``if'' direction is clear, noting the epimorphism from $BS(2m,2m) $ to $BS(1,-1)  $ mapping $t$ to $a'$ and $a$ to $t'$. 

For the converse, 
 we first suppose that $BS(m,n)$ and $ BS(m',n')$ are non-elementary (i.e.\ different from $\Z^2$ and $K$). If $f$ exists, it is elliptic-friendly  by Corollary \ref{ef}, 
 so $\frac mn$ equals $\frac {m'}{n'}$ or $\frac {n'}{m'}$ by Lemma \ref{modeg}. If $m\ne n$, abelianizing  shows that $m-n$ is divisible by $m'-n'$. If $m=n$, dividing by the center $\langle a^m\rangle$ shows that $m$ is divisible by $m'$. The result follows. 

The case when $G$ is elementary is easy and left to the reader (any epimorphism from $G$ to a non-cyclic GBS group is an isomorphism), so assume that $G'$ is elementary and $G$ is not. 
    If $G'=\Z^2$, Lemma \ref{modeg} implies $m=n$ as required. If $G'=K$, it implies   $m=\pm n$ and we have to rule out the possibility that $m=n$ with $m$ odd. 

  Assume $m=n$ and consider   $\chi: 
  G'\twoheadrightarrow   \Z/2\Z\times\Z/2\Z$ obtained by abelianizing and killing the image of the center $\langle t'{}^2\rangle$. Since $a^m$ is central in $BS(m,m)$, its image by $f$    is killed by $\chi$. If $m$ is odd, the element   $a$ itself is   killed by $\chi\circ f$. This is a contradiction since $\chi\circ f$ is onto, but $G/\langle\langle a \rangle\rangle\simeq\Z$ does not map onto $\Z/2\Z\times\Z/2\Z$.
\end{proof}

\subsection{Baumslag-Solitar quotients of GBS groups}

\begin{lem} \label{qbs} Let $G$ be a non-elementary GBS group, and write $\beta =\beta (G)$. 
\begin{itemize} 
\item 
If $\beta =0$, then $G$ has no quotient isomorphic to a Baumslag-Solitar group, except possibly to    $K=BS(1,-1)$.
\item If $\beta \ge1$, then $G$ has a quotient isomorphic to a Baumslag-Solitar group.   There exists an elliptic-friendly epimorphism $f$ from $G$ to a Baumslag-Solitar group   if and only if the image of $\Delta_G$ is cyclic (possibly trivial).

\end{itemize}
\end{lem}

\begin{proof}
Let  $\Gamma$ be any labelled  graph representing $G$.

 If $\beta =0$, then $G$ may be   generated by a set of elements which are all commensurable with each other, and any quotient of $G$ inherits  this property. This rules out all  Baumslag-Solitar groups except  $K=\langle u,v\mid u^2=v^2\rangle$.
 
 Now suppose  $\beta =1$. The image of $\Delta_G$ is cyclic, and we construct an elliptic-friendly epimorphism   
 as follows.  The graph $\Gamma$  contains a unique embedded circle. Choose an edge on this circle, and collapse all other edges to a point. This  expresses $G$ as an HNN extension where the base group $G_0$ is a GBS group represented by  a labelled graph which is a tree, and the amalgamated subgroups are cyclic and consist of elliptic elements of $G_0$. There is an elliptic-friendly epimorphism from $G_0$ to $\Z$  (see Proposition 3.3 of  \cite{Le}). It extends to  an elliptic-friendly epimorphism from $G$ to some $BS(m,n)$.
 
 If $\beta \ge 2$, then $G$ maps onto $F_2$ hence onto any Baumslag-Solitar group.
If there is  an elliptic-friendly epimorphism $f:G\twoheadrightarrow   BS(m,n)$, the image of $\Delta_G$ is cyclic by Lemma \ref{modeg}. 
 
 Conversely, we suppose that the image of $\Delta_G$ is cyclic, contained in the subgroup of $\Q^*$ generated by $\frac mn$ (with $m$ and $n$ coprime), and we now construct an elliptic-friendly epimorphism $f:G\twoheadrightarrow   BS(m,n)$ as in the case $\beta =1$.   We choose a maximal subtree $\Gamma_0\inc \Gamma$ and we map the GBS group $G_0$ represented by  $\Gamma_0$ to $\Z$ without killing the elliptic elements. This yields a quotient $G'$ of $G$ with presentation $\langle a,t_1,\dots,t_\beta \mid t_ia^{m_i}t_i\m=a^{n_i}\rangle$. Each $\frac{m_i}{n_i}$ is in the image of $\Delta_G$, so 
 we may write $(m_i,n_i)=(\theta_im^{\kappa_i}, \theta_in^{\kappa_i})$ or $(m_i,n_i)=(\theta_in^{\kappa_i}, \theta_im^{\kappa_i})$ 
 with integers 
   $\theta_i\ne0$ and $\kappa_i\ge0$.  We now map $G'$ to $BS(m,n)=\langle a,t\mid ta^mt\m=a^n\rangle$ by sending $a$ to $a$ and $t_i$ to $t^{\pm\kappa_i}$. This is onto because the image of $\Delta_G$ in $\Q^*$ is generated by the numbers $\frac{m_i}{n_i}$.
 \end{proof}
 
 \begin{example}
  Some GBS  groups with $\beta =0$ map onto $K$, some do not. For instance, let us show that \emph{$G=\langle a,b\mid a^m=b^n\rangle$ maps onto $K$ if and only if $m$ and $n$ are even and divisible by the same powers of 2 (i.e.\ $m=2\lambda r$ and $n=2\lambda s$ with $r,s$ odd and coprime)}. 
  
  We   view $K$ as $\langle u,v\mid u^2=v^2\rangle$. If $m=2\lambda r$ and $n=2\lambda s$ as above, we map $G$ to $K$ by sending $a$ to $u^s$ and $b$ to $v^r$ (to see that $u^s$ and $v^r$ generate $K$, write $x r+y s=1$ and observe that $x$ or $y$ must be even).
  
  Conversely, suppose $f$ is an epimorphism from $G$  to $K$. Since $a^m$ is central, $f(a^m)$ is a power of $u^2$, so $f(a)$ is conjugate to a power of $u$ or $v$, and similarly for $f(b)$. Abelianizing shows that $f(a)$ is conjugate to an odd power of $u$, and $f(b)$ to an odd power of $v$ (up to swapping $u$ and $v$). In particular, $m$ and $n$ are even. The relation $a^m=b^n$ forces $m$ and $n$ to be divisible by the same powers of 2.
  \end{example}   
  
 \begin{example}
 It is not true that every GBS group with $\beta =1$ maps onto a Baumslag-Solitar group different from $K$. For instance, $K$ is the only Baumslag-Solitar quotient of 
 $G=\langle  a,b,t\mid a^2=b^2, tbt\m=b\m \rangle$: if $G$ maps onto $BS(m,n)$,    Lemma \ref{modeg} implies $m=-n$;   abelianizing shows $m=\pm1$.
  \end{example}
  
 \subsection{Descending chains}       

Since Baumslag-Solitar groups may fail to be Hopfian, they do not satisfy the descending chain condition: there   exists an infinite sequence of non-injective epimorphisms $f_n:G_n\twoheadrightarrow  G_{n+1}$. On the other hand, Baumslag-Solitar groups satisfy the descending chain condition  if we require $G_n$ and $G_{n+1}$ to be non-isomorphic (see Lemma \ref{eBS}). 

Among GBS groups, the descending chain condition fails in the strongest possible way. 
\begin{prop}
There exists a sequence $(f_n) _{n\ge1}$ of epimorphisms $f_n:G_n\twoheadrightarrow  G_{n+1}$  such that:
\begin{itemize}
\item each $G_n$ is a GBS group which is a quotient of $BS(18,36)$ and maps onto $BS(9,18)$;
\item there is no epimorphism $G_{n+1}\twoheadrightarrow  G_n$.
\end{itemize}
\end{prop}
In particular,  $G_n$ and $G_{n+1}$ are  not isomorphic, and not even  epi-equivalent.

\begin{proof}
Let  $G_n=\langle a,b,t\mid a^6=b^{2^n}, tb^{3}t\m=b^6\rangle$   (it is represented by a lollipop as on Figure \ref{ptit} with $Q=6$). It is generated by $a$ and $t$, which satisfy $ta^{18}t\m=a^{36}$, so is a quotient of $BS(18,36)$.  
Constructing epimorphisms from $G_n$ to $G_{n+1}$, and from $G_n$ to $BS(9,18)$,  is easy and left to the reader. We assume that there is an epimorphism $f:G_{n+1}\twoheadrightarrow  G_n$, and we argue towards a contradiction.

We denote by $a',b',t'$ the generators of $G_{n+1}$.  
Since $(f(a'),f(t'))$ is a generating pair of  $G_n$ with $f(a')$ elliptic, Lemma \ref{lesp} and Remark \ref{rr} imply that $f(a')$ 
 is (conjugate to) $ a^r$ with $r$ coprime with 6. Since $(a'){}^6=(b')^{2^{n+1}}$, the element $c= f((a')^6)=a^{6r}=b^{2^nr}$ has a $2^{n+1}$-th root $u$ in $G_n$. The group $A=\langle a,b\rangle$ is a vertex group in  a (non-GBS) splitting of $G_n$, with incident edge groups conjugate to $\langle b^3\rangle$ and $\langle  b^6\rangle$. Since $c=b^{2^nr}$ is central in $A$, and $r$ is not divisible by 3, $c$ fixes a unique point in the Bass-Serre tree of this splitting, so $u\in A$. 
Now consider the epimorphism from $A$ to $\Z$ mapping $a$ to $2^{n-1}$ and $b$ to 3. The image of $c$ is $3\,r\,2^n$, so $c$ does not have a $2^{n+1}$-th root   in $A$ since $r$ is odd. 
We have reached the desired  contradiction.
\end{proof}

\section{Maps from Baumslag-Solitar groups}      \label{sens1}   

We have seen (Lemma \ref {qdeBS}) that any 2-generated GBS group $G$ is a quotient of some $BS(m,n)$.
In this section we fix $G$ and we determine for which values of $m$ and $n$ this happens.

\begin{thm}\label{deBS}
 Let $G$ be a     2-generated GBS group, represented by  a reduced labelled graph $\Gamma$ as on Figure \ref{disp}. Assume that $G$ is not cyclic or the Klein bottle group $K$.
 \begin{itemize}
 \item If $\Gamma$ is a segment, then $G$ is a quotient of $BS(m,n)$ if and only if $m=n$, and $m$ is divisible by $Q$ or $R$. 
 \item    If $\Gamma$ is  a lollipop (possibly a circle), then $G$ is a quotient of $BS(m,n)$ if and only if $(m,n)$ is an integral multiple of $(QX,QY)$ or $(QY,QX)$.
\end{itemize}
\end{thm}

In the lollipop case, there is a unique ``smallest''  Baumslag-Solitar group mapping onto $G$, namely $BS(QX,QY)$. In the segment case, there are in general two   groups,   $BS(Q,Q)$ and $BS(R,R)$.

We establish a   lemma before starting the proof.

\begin{lem}\label{wt2}

 Let $\Gamma$ be a   labelled graph homeomorphic to  a segment, as on Figure \ref{disp}, with associated group $G=\langle a_0,\dots, a_k\mid  (a_i)^{q_i}=(a_{i+1})^{r_{i+1}}, i=0,\dots,k-1\rangle$. Given $r_0\ne0$, let $(a_0)^{r_0N}$be a generator of 
$ \langle (a_0)^{r_0}\rangle\cap\langle a_k\rangle$. Then:

 \begin{enumerate}
 \item $(a_0)^{q_0\cdots q_{j-1} }=(a_j)^{r_1\cdots r_{j} }$ for $j=1,\dots, k$. In particular,  $N$ divides $q_0\cdots q_{k-1}$. 
 \item  If 
  $q_j$ and $r_i$ are coprime whenever $0\le i\le j\le k-1$,  
   then $N=\pm q_0\cdots q_{k-1}$.
  
 \item 
  If  $p$ is a prime which divides no $ q_j\wedge r_i $ with   $0\le i\le j\le k-1$, then $N=q_0\cdots q_{k-1}/  \theta$ with $ \theta$ not divisible by $p$.
  \item Let $p$ be prime, $\alpha\ge0$, and $r_0=1$. Suppose that each $q_j$ for $0\le j\le k-1$ and each $r_i$ for $1\le i\le k$ is divisible by $p^\alpha$ but not by $p^{\alpha+1}$. Then $N$ is   divisible by $p^\alpha$ but not by $p^{\alpha+1}$.
  
\end{enumerate}
\end{lem}

The fourth assertion will be used only in Subsection \ref{ebd}.

\begin{rem*} $\langle a_0 \rangle\cap\langle a_k\rangle$ is the center of $G$.
\end{rem*}

\begin{proof}
 For the first assertion, we simply 
 use the relations $(a_i)^{q_i}=(a_{i+1})^{r_{i+1}}$.  We show   3 (which clearly implies  2)  by induction on $k$. 
 Let $\theta_j\in\Z$ be such that $ \langle (a_0)^{r_0}\rangle\cap\langle a_j\rangle$ is generated by $(a_0)^{r_0q_0\cdots q_{j-1}/\theta_j}=(a_j)^{r_0r_1\cdots r_{j} /\theta_j}$, with $\theta_0=1$. We must show that $\theta_k$ is not divisible by $p$.

 Using $  \langle (a_0)^{r_0}\rangle\cap\langle a_k\rangle=\langle (a_0)^{r_0}\rangle\cap   \langle  (a_{k-1})^{q_{k-1}}\rangle$, and the formula $m\Z\cap n\Z=\frac{mn}{ m\wedge n }\Z$,
 we    
  write
$$ \langle (a_0)^{r_0}\rangle\cap\langle a_k\rangle=\langle (a_0)^{r_0}\rangle\cap \langle  a_{k-1} \rangle\cap \langle  (a_{k-1})^{q_{k-1}}\rangle
=\langle  (a_{k-1})^{r_0\cdots r_{k-1}/\theta_{k-1}}\rangle \cap\langle  (a_{k-1})^{q_{k-1}}\rangle$$
$$
=\langle  (a_{k-1})^{\frac{r_0\cdots r_{k-1}q_{k-1}}{\theta_{k-1} \theta'}}\rangle
=\langle  (a_{k})^{r_0\cdots r_{k-1} r_{k}/\theta_{k-1}\theta'}\rangle
$$ 
with $\theta'$ the gcd of $r_0\cdots r_{k-1} /\theta_{k-1}$ and $q_{k-1}$.

If $p$ is as in Assertion 3, then $\theta_{k-1}$ is not divisible by $p$ by induction, and neither is $\theta'$, 
so $\theta_k=\pm\theta_{k-1}\theta'$ is not divisible by $p$.

 Assertion 4 is equivalent to saying that the largest power of $p$ dividing $\theta_k$ is $p^{(k-1)\alpha}$. We prove this by induction on $k$, noting that $\theta_1=1$. By induction the largest power of $p$ dividing  $r_0\cdots r_{k-1} /\theta_{k-1}$   is $p^\alpha$, so the same is true for $\theta'$ and the result follows since  $\theta_k=\pm\theta_{k-1}\theta'$.
\end{proof}

The same proof shows:

\begin{lem}\label{wt}

Let $v_0,\dots,v_k$ be a segment of length $k$ in the Bass-Serre tree $T$ associated to some labelled graph. Let $a_i$ be a generator of  the stabilizer of $v_i$, let $c_i$ be a generator of  the stabilizer of the edge $v_iv_{i+1}$, and define $q_0,\dots, q_{k-1},r_1,\dots r_k$, by $c_i=(a_i)^{q_i}=(a_{i+1})^{r_{i+1}}$. 
Given $r_0\ne0$, 
 the   conclusions of Lemma \ref{wt} hold for the index $N$ of $ \langle (a_0)^{r_0}\rangle\cap\langle a_k\rangle$
in $ \langle (a_0)^{r_0}\rangle $. \qed
\end{lem}

We can now prove the main result of this section.

\begin{proof}[Proof of Theorem \ref{deBS}] 
We assume that $G$ is not a Baumslag-Solitar group (the theorem follows from Lemma \ref{eBS} if it is). In particular, $G$ is non-elementary, and so is any $BS(m,n)$ which maps onto $G$.
We fix a presentation $\langle a,t\mid ta^mt\m=a^n\rangle$ for $BS(m,n)$.

$\bullet$ First suppose that $\Gamma$ is  a segment. Since $G$ is 2-generated, it is generated by $a_0$ and $a_k$. These elements satisfy $(a_0)^Q=(a_k)^R$ by Lemma \ref{wt2}, so $G$ is a quotient of $BS(\alpha Q,\alpha Q)$ and $BS(\alpha R,\alpha R)$ for any integer $\alpha$: send $a$ to $a_0$ and $t$ to $a_k$, or $a$ to $a_k$ and $t$ to $a_0$.

Conversely, suppose that $f$ is an epimorphism from $BS(m,n) $ to $G$.  It is elliptic-friendly  by Corollary \ref{ef}, so $m=n$ by Lemma \ref{modeg} since $G$ has trivial modulus.  
The element $f(a)$ is elliptic,   
so by Lemma \ref{lesp} we may assume (by composing $f$ with an inner automorphism of $G$)  that $f(a) $ is a power 
of $a_0$ or $a_k$, say $f(a)=(a_0)^{r_0}$ with $r_0\ne0$.  In this case we show that $m$ is a multiple of $Q$ (it is a multiple of $R$ if $f(a)$ is a power of $a_k$).

We will apply Assertion 2 of Lemma \ref{wt2}, so
we  now check that  $q_j$ and $r_i$ are coprime for $0\le i\le j\le k-1$:  
this follows from Lemma \ref{copr} for $i>0$  (note that $r_0$ is not a label of $\Gamma$), 
and from Remark \ref{rr} for $i=0$.

The element $a^m$ is central in $BS(m,m)$, so $(a_0)^{r_0m}$ is central in $G$. In particular, it must be a power of $a_k$ (since $ | r_k | >1$, the centralizer of $a_k$ is $\langle a_k\rangle$). 
Thus $(a_0)^{r_0m}$ belongs to 
 $ \langle (a_0)^{r_0}\rangle\cap\langle a_k\rangle$, which  is generated by $(a_0)^{r_0Q}$ by Lemma \ref{wt2}, so $m$ is a multiple of $Q$.

$\bullet$ We now suppose that $\Gamma$ is a lollipop. We use the standard presentation   $G=\langle a_0,\dots, a_{k-1}, b_0,b_1,\dots, b_{\ell-1},\tau\rangle$
described in Subsection \ref{2g}.
Since $G$ is 2-generated, it is generated by $a_0$ and $\tau$
 (by $b_0$ and $\tau$ if $\Gamma$ is a circle). To unify notation, we will   write $g_0$ for $a_0$ if $\Gamma$ is not a circle,  for $b_0$ if $\Gamma$ is a circle, and we denote by $z_0$   the associated vertex $v_0$ or $w_0$.  

 The first assertion of Lemma \ref{wt} yields $(a_0)^Q=(b_0)^R$ and $\tau(b_0)^X\tau\m=(b_0)^Y$, so that $\tau(g_0)^{QX}\tau\m=(g_0)^{QY}$. It follows that $G$ is a quotient of $BS(QX,QY)$.
 
 Conversely, consider an epimorphism $f:BS(m,n)\twoheadrightarrow  G$.  It is elliptic-friendly  by Corollary \ref{ef}, and  $\frac mn$ equals $\frac XY$ or $\frac YX$ by Lemma \ref{modeg}. Also note that $(f(a),f(t))$ is a   generating pair  of $G$, with $f(a)$ elliptic.
 
By Lemma \ref{lesp}, the element $f(a)$ belongs, up to conjugacy, to the group carried by  a vertex belonging to every plateau,  so we may assume (by composing with an inner automorphism, and renumbering vertices if $\Gamma$ is a circle) that 
  $f(a)=(g_0)^{r_0}$ with $r_0\ne0$ (as above, beware that $r_0$ is not a label of $\Gamma$). 
  By Remark \ref{rr}, $r_0$ and $QX\wedge QY$ are coprime. 
 
  Killing all elliptic elements defines an infinite cyclic quotient $Z$ of $G$, which is the topological fundamental group of $\Gamma$.  The elements $a_i,b_j,f(a)$ vanish in $Z$, while $\tau$ and $f(t)$ map to generators. We may assume that $\tau$ and $f(t)$ map to the same generator. This implies $\frac mn=\frac XY$. We have to show that $(m,n)$ is an integral multiple of  $(QX,QY)$.

 We   consider the Bass-Serre tree $T$ of $\Gamma$. We also consider the topological universal covering $\tilde\Gamma$, a line with segments of length $k$ attached (see Figure \ref{arbs}). The group $G$ acts on $\tilde \Gamma$ through $Z$. There is a natural equivariant map $\varphi:T\to \tilde\Gamma$ sending   vertex to vertex and edge to   edge; composing $\varphi$ with the covering map $\tilde\Gamma\to\Gamma$  yields the quotient map $\pi:T\to\Gamma=T/G$.
 
 \begin{figure}[h]
\begin{center}
\begin{pspicture}(100,310)

\rput(0,230){

\psline(-150,0)(-70,0)

\psline(130,-120)(190,-120)
\pscircle(220,-120){30}
\psline (-70,0)(-70,-120)
\psline(-150,-120)(-70,-120)

\rput(10,0){
\psline(60,20)(60,-140)
\psline(-10,-120)(60,-120)
\psline(-10,0)(60,0)
\psline[linestyle=dashed](60,20)(60,60)
\psline[linestyle=dashed](60,-140)(60,-180)

  \pscircle*( -10,0) {2}
   \pscircle*(-10,-120) {2}
   
   \rput(-10,-10) {$\varphi( u)$}  
     \rput(-10,-110) {$\varphi( u')$}    
}

 \pscircle*( -150,0) {2}
  \pscircle*( -70,0) {2}
   \pscircle*(-70,-120) {2}
    \pscircle*( -115,-120) {2}
     \pscircle*( -150,-120) {2}
\pscircle*( -150,-120) {2}
\psline( -105,-120)(-105,-145)
\pscircle*(-105,-145) {2}
   \pscircle*(-70,-35) {2}
   \pscircle*(-70,-85) {2}
   \pscircle*(-120,0) {2}
   \pscircle*(130,-120){2}

\rput(-150,-10) {$ u$}

   \rput(-150,-110) {$ u'$}  
     \rput(-115,-110) {$ u''$}  
       \rput(-65,-128) {$ \ov u$}  
     \rput(-95,-145) {$gu'$}  
      \rput(131,-132) {$ \pi(u)$}     
  
  {\blue
  
\rput(-143,8) {$ q_0$}  
    \rput(-125,8) {$ r_1$}
    \rput(-77,8) {$ r_k$}
     \rput(-60,-10) {$ x_0$}  
     \rput(-60,-110) {$  y_{\ell}$}  
     \rput(-60,-28) {$  y_{ 1}$}     
          \rput(-60,-43) {$  x_{ 1}$}  
          \rput(-55,-93) {$  x_{\ell-1}$}  
     
  }

 \rput(190,-220) {$ \Gamma$}   
 \rput(40,-220) {$\tilde\Gamma$}   
 \rput(- 110,-220) {the segment $uu'\inc T $}

}

        \end{pspicture}
\end{center}
\caption{}
\label{arbs}
\end{figure}

Since $G$ is not a solvable Baumslag-Solitar group, no label near the vertex $z_0$ of $\Gamma$ carrying $g_0$ equals $\pm1$, so $g_0$ fixes a 
  unique point $u$ in  $T$.  This point is fixed by $f(a)=(g_0)^{r_0}$. 
  We have $ t\m a^n =a^mt\m$, so $f(a^m)=(g_0)^{mr_0}$ fixes a point $u'\in T$ such that $\varphi(u')=f(t\m)\varphi(u)$, for instance $f(t\m)u$. Among all such points $u'$, we choose one at minimal distance from $u$;
  note that $u'$ necessarily belongs to the $G$-orbit of $u$. 
    We claim that $\varphi$ is then injective on the segment $uu'$ (i.e.\  the map from $uu'$ to $\Gamma$ is locally injective).

If not, there is a point $u''\in T$ between $u$ and $u'$ such that the initial edges $e$ and $e'$ of $u''u$ and $u''u'$ have the same projection in $\Gamma$, so some $g\in G$ in the stabilizer of $u''$ maps $e'$ to $e$. Now consider $gu'$. It is closer to $u$ than $u'$, and it has the same image as $u'$ in $\tilde \Gamma$  since $g$ acts as the identity on $\tilde\Gamma$ (it is elliptic, so vanishes in $Z$). The element $(g_0)^{mr_0}$ fixes $u$ and $u'$, hence $u''$, so it commutes with $g$ and therefore fixes $gu'$. This contradicts the choice of $u'$ and proves the injectivity claim.

Since  $\varphi(u')=f(t\m)\varphi(u)=\tau\m\varphi(u)$, the projection of $uu'$ to $\Gamma$   is a loop going once around the circle, so it is the immersed path  of length $\ell+2k$ $$v_0, \dots,v_{k-1},v_k=w_0,\dots,w_{\ell-1}, w_\ell=w_0=v_k,v_{k-1}, \dots, v_0.$$ We wish to apply Lemma \ref{wt} to the segment $uu'$, knowing  that $g_0$ generates the stabilizer of $u$  and $(g_0)^{mr_0} $ fixes the whole segment. Recall that we want $(m,n)$ to be  a multiple of $(QX,QY)$.
 
First suppose that  $\Gamma$ is   a circle (so $uu'$ has length $\ell$, and $g_0=b_0$). We know that $x_j$ and $y_i$ are coprime for $1\le i\le j\le \ell-1$   
by Lemma \ref{copr}, but  
$r_0$ does not have to be  coprime with $X$,   
so we cannot apply the second assertion  of  Lemma \ref{wt}: we will need the third one.  

Our goal  is to show that $\frac mX=\frac nY$ is an integer (since $Q= 1$). If it is not, write it in lowest terms and consider a prime $p$ dividing the denominator. It divides $X$ and $Y$, so it does not divide $r_0$, as explained above. Assertion 3 of  Lemma \ref{wt} then implies that $m$ is a multiple of $X/\theta$ with $\theta$ not divisible by $p$, a contradiction. 

In the lollipop case, we consider $p$ dividing the denominator of  $\frac m{QX}=\frac n{QY}$. We will    apply Lemma \ref{wt} to the initial segment $u\ov u$ of length $k+\ell$ of $uu'$. 
We know that 
 $q_j$ and $r_i$ are coprime for $0\le i\le j\le k-1$ (recall that $r_0$ and $Q$ are coprime), and    $x_j$, $y_i$ are coprime for $1\le i\le j\le \ell -1$, but we do not know that  $ r_i$ and $ x_j $ are coprime.  
 
 If  $p$ does not divide 
 $ r_i\wedge  x_j $ for $0\le i\le k$ and $0\le j\le\ell-1$, Assertion 3 of  Lemma \ref{wt} applies and gives a contradiction as above: $m$ is a multiple of $QX/\theta$ with $\theta$ not divisible by $p$. 
  
 The condition that $p$ does not divide  $ r_i\wedge  x_j $ certainly holds if 
 $p$ does not divide $X$.
 By symmetry of the statement of Theorem \ref{deBS}, we may therefore assume that  $p$ divides both $X$ and  $Y$. It follows  that $p$ does not divide $r_0$, and also that it does not divide $R=\prod_{i=1}^k r_i$ by the second assertion of Lemma \ref{copr}. Thus the condition also  holds in this case. 
\end{proof}

\section{GBS quotients of $BS(m,n)$}     \label{cons}

In the previous section we fixed $G$ and we determined for which values of $m,n$ one may view $G$ as a quotient of $BS(m,n)$. We now fix $m,n$
and we study the set of all GBS groups $G$ which are quotients of $BS(m,n)$, using Theorem \ref{deBS}.   
We assume $G\ne BS(1,\pm1)$ (the results are trivial in these two cases). 

We always represent $G$ by a reduced labelled graph $\Gamma$. 

\begin{prop} Let $\Gamma$ be a reduced labelled graph   representing a GBS group $G$ which is a quotient of $BS(m,n)$.

\begin{enumerate}
\item 
There is a bound, depending only on $m$ and $n$, for the number of edges of $\Gamma$. 
\item If $m\ne n$, and $G\ne K$, every prime  $p$ dividing a label of $\Gamma$ must divide $m n$.
\end{enumerate}
\end{prop}

\begin{proof}
This follows  easily  from Theorem \ref {deBS}, noting that at most 2 labels of $\Gamma$ may be equal to $\pm1$. If $\Gamma$ is a lollipop, $(m,n)$ is a multiple of $(QX,QY)$ or $(QY,QX)$. 
In particular $Q,X,Y$  are bounded, so the number of edges is bounded. Primes dividing $Q$, $X$ or $Y$ divide $m$ or $n$, and primes dividing $R$ divide 
$X$ or $Y$ by Lemma \ref{copr}, so Assertion 2 holds in this case. If $\Gamma$ is a segment, then $m=n$ and the number of edges is bounded because $Q$ or $R$ divides $m$.
\end{proof}

 \begin{rem} \label{bd}
 In the segment case, one of the numbers $Q$ or $R$ is bounded in terms of $m,n$, but there is no control on the other. In the lollipop case, $Q,X,Y$  are bounded, but $R$ does not have to be. See also Remark \ref{bdd}.
 \end{rem}
 
 \begin{prop} \label{quis}
Given $m$ and $n$, the following are equivalent    (up to swapping $m$ and $n$):
\begin{enumerate}
 \item every non-solvable GBS quotient of $BS(m,n)$ is isomorphic to $BS(m,n)$;
 \item $m=\pm1$, or $ | m | $ is a prime number and $m\ne n$. 
\end{enumerate}
\end{prop}

\begin{proof} 
2 implies 1 by 
Theorem \ref{deBS} and Lemma \ref{padun}. We now construct a non-solvable quotient under the assumption that 2 does not hold. In particular, $m,n\ne\pm1$.  If $m=n$, we map $BS(m,n)$ to $\langle a,b\mid a^m=b^N\rangle$ with $N$ arbitrary, so assume $m\ne n$.  Then none of $ | m | , | n | $ is prime. 
If $p$ is a prime dividing $m$ and $n$,
then $BS(m,n)$ maps onto $BS(m/p,n/p)$ (see Lemma \ref{eBS}). If $m,n$ are coprime, we write  $m=\alpha\beta$ and $n=\gamma\delta$ with integers $\alpha,\beta, \gamma, \delta$ different from $\pm1$, and  we consider $G =\langle a,b,t\mid a^\alpha=b^\gamma, tb^\beta t\m=a^\delta\rangle$ (it is represented by a circle of length 2). It is generated by $a$ and $t$, which satisfy $ta^{\alpha\beta}t\m=a^{\gamma\delta}$, so it is a quotient of 
$BS(m,n) $.
It is not isomorphic to $BS(m,n)$ by \cite{Fodef} since $ \alpha\wedge \delta =  \beta\wedge\gamma =1$  ($\Gamma$ is strongly slide-free).
\end{proof}

\begin{cor} \label{qui}
Given $m$ and $n$, the following are equivalent    (up to swapping $m$ and $n$):
\begin{enumerate}
 \item every non-cyclic GBS quotient of $BS(m,n)$ is isomorphic to $BS(m,n)$;
 \item $m=\pm1$, or $ | m | $ is a prime number not dividing $n$. 
\end{enumerate}
\end{cor}

\begin{proof}  
As before, 2 implies 1 by 
Theorem \ref{deBS} and Lemma \ref{padun}. For the converse, we simply observe that $BS(m,n)$ maps onto $BS(1,\frac nm)$ if   $ |m |  $ is a prime dividing $n$. 
\end{proof}

\begin{cor} If $G$ is a Baumslag-Solitar group, the following are equivalent:
\begin{itemize}
\item every epimorphism from $G$ to a non-cyclic GBS group is an isomorphism;
\item $G$ is a solvable $BS(1,n)$. 
 \end{itemize}
 \end{cor}
 
 This is a special case of Theorem A of  \cite{DRT}.
\begin{proof}
 This follows from the previous corollary: $BS(1,n)$ is Hopfian, but $BS(m,n)$ is not when $m$ is a prime not dividing   $n\ne\pm1$ (see \cite{BS,CL} and Remark \ref{nh}).
\end{proof}

\begin{prop} \label{infq}
Given $m$ and $n$, 
 the  following are equivalent:
\begin{enumerate}
\item
if  a GBS quotient $G$ of $BS(m,n)$ is represented by a labelled graph $\Gamma$ with no label $\pm1$, and $G\ne K$, then $\Gamma$   is homeomorphic to a circle;
\item
up to    isomorphism, there are only  finitely many GBS quotients of $BS(m,n)$;
\item one of the following holds   (up to swapping $m$ and $n$):
\begin{enumerate}
\item $m$ and $n$ are coprime;
\item $ | m | $ is prime and $m\ne n$;
\item $m=-n$;
\item $ | m | $ and $ | n | $ are powers of the same prime $p$, and $m\ne n$.
\end{enumerate}
\end{enumerate}
\end{prop}

In particular:
\begin{cor}
 If $BS(m,n)$ has a GBS quotient $G\ne K$ represented by a labelled graph $\Gamma$ with no label equal to $\pm1$ such that $\Gamma$   is not homeomorphic to a circle, then $BS(m,n)$ has infinitely many non-isomorphic GBS quotients. \qed
\end{cor}

\begin{rem}\label{bdd}
When none of $m,n$ divides the other, reduced labelled graphs representing a quotient of $BS(m,n)$
have no label 
  $\pm1$, and there are finitely many such graphs. 
But when divisibility occurs it is important in the proposition and the corollary to exclude labels equal to $\pm1$. For instance, $BS(2,4)$ has finitely many GBS quotients up to isomorphism, but it is represented by the (reduced) 
non-circle lollipop of Figure \ref{ptit} with $R$ an arbitrarily large power of 2 (and $Q=Y=2,X=1$).
\end{rem} 

\begin{proof}[Proof of Proposition \ref{infq}]
Lemma \ref{padun} and Remark \ref{bd} show that 1 implies 2. We show that 3 implies 1.
Segments are ruled out since each assumption implies $m\ne n$ (unless $m=n=\pm1$). Non-circle lollipops are ruled out  under the first two assumptions. Under the other two, $X$ and $Y$ are opposite or powers of the same prime (up to sign), so lollipops are ruled out by the last assertion of Lemma \ref{copr}. 

To show that 2 implies 3, we assume that
  none of the   conditions (a)-(d) holds, and we construct infinitely many quotients. To prove that they are  not isomorphic, we use a result  by 
  Clay-Forester \cite{CFWh} (see Section 2.2 of \cite{CFIso} for the case of GBS groups):  
if two reduced labelled graphs representing   non-elementary GBS groups 
cannot be connected by a finite sequence of admissible sign changes, slide moves, induction moves, $\cala^{\pm1}$ moves, they represent non-isomorphic   groups. 
  
  We distinguish three cases. 
  
  If $m=n\ne\pm1$, then $BS(m,m)$ maps onto $\langle a,b\mid a^m=b^N\rangle$ for any $N$. 
  
    If none of $m,n$ divides the other, we write $m=\delta m'$ and $n=\delta n'$, with $m',n'$ coprime and $\delta,m',n'\ne\pm1$. Let $p$ be a prime dividing $n'$. For $N> 1$ such that $p^N$ does not divide $n'$, consider $$G_N=\langle a,b,t\mid a^ \delta=b^{p^N}, tb^{m'}t\m=b^{n'}\rangle $$ as on Figure \ref{ptit}. These groups are quotients of $BS(m,n)$ by the easy direction of Theorem \ref{deBS}.   Using \cite{CFWh}, one checks   that  they are pairwise non-isomorphic: no sequence of moves     can transform the graph defining $G_N$ into the graph defining $G_M$ for $M\ne N$. 

    If $m$ divides $n$, and $m\ne n$, we may write $m=\alpha\beta$ and $n=\alpha\beta\gamma\delta$ with $\alpha,\delta$ coprime and $\alpha,\beta,\delta\ne\pm1$ (but $\gamma=\pm1$ is allowed). We define $$H_N=\langle a,b,t\mid a^ \beta=b^{\delta^N}, tb^{\alpha}t\m=b^{\alpha\gamma\delta}\rangle $$
    for $N>1$ such that $\delta ^N$ does not divide $\gamma\delta$.
  These groups are pairwise non-isomorphic    by \cite{CFWh}.
\end{proof}

\section{Groups epi-equivalent  to a Baumslag-Solitar group}   \label{sens2}

In Section \ref{sens1} we have determined for which values of $m,n$ a given      2-generated GBS group $G$ is a quotient of $BS(m,n)$. We now consider Baumslag-Solitar  quotients of $G$. 
Lemma \ref{qbs} says that we should restrict to groups represented by a lollipop (possibly a circle). In this case Theorem \ref{deBS} implies that there is a   
``smallest'' Baumslag-Solitar group mapping to $G$, namely $BS(QX,QY)$, and we shall determine whether $G$ maps onto $BS(QX,QY)$ or not (Theorem \ref{hop}).     In other words, we shall find the GBS groups epi-equivalent to a given $BS(m,n)$, in the following sense.

\begin{dfn} Two groups are \emph{epi-equivalent} if each is isomorphic to a quotient of the other. 
\end{dfn}

Of course, any group epi-equivalent to a Hopfian group is isomorphic to it. Recall \cite{CL} that  $BS(m,n)$ is Hopfian if and only if $m$ and $n$ have the same prime divisors, or one of them equals $\pm1$ (see Remark \ref{nh}).

\begin{rem}
 When a 2-generated GBS group   $G$ represented by a lollipop does not map onto $BS(QX,QY)$ (in particular when $BS(QX,QY)$ is Hopfian), there does not always exist a maximal Baumslag-Solitar quotient. For instance, $\langle a,b,t \mid a^3=b^3, tb^2t\m=a^4\rangle$ is a  quotient of the Hopfian group $BS(12,6)$. It maps onto $BS(4,2)$ and  $BS(6,3)$, but not onto $BS(12,6)$.
 \end{rem}

We start with a simple case.

\begin{prop}\label{prem}
Fix coprime integers $m$ and $n$.
  Given a non-cyclic  GBS group $G$, the following are equivalent:
 
\begin{enumerate}
 \item $G$ is a quotient of $BS(m,n)$;
 \item $G$ is epi-equivalent to $BS(m,n)$;
 \item $G$  
 may be represented by  a reduced labelled graph $\Gamma$ homeomorphic to a circle, with $X=m$ and $Y=n$.
\end{enumerate}
Up to isomorphism, there are only finitely many groups $G$ satisfying these conditions.
\end{prop}

\begin{proof}
 Clearly 2 implies 1. It follows from Theorem \ref{deBS} that 1 implies 3. If 3 holds, $G$ is 2-generated 
 by Theorem 1.1 of \cite{Le2} (there is no proper plateau in $\Gamma$), and is
  a quotient of $BS(m,n)$ by the easy direction of Theorem \ref{deBS}. Lemma \ref{qbs} yields an  epimorphism from $G$  to some Baumslag-Solitar group. Since $m$ and $n$ are coprime, this group must be $BS(m,n)$ by Lemma \ref{eBS}, so 2 holds. The finiteness statement follows from Proposition \ref{infq}.
\end{proof}

\begin{rem}
When $m$ and $n $ are not coprime, and $BS(m,n)$ is not Hopfian, Corollary \ref{iee} will show that there exist infinitely many groups epi-equivalent to $BS(m,n)$, provided none of $ | m | , | n | $ is prime. 
\end{rem}

\begin{cor} \label{pointi} A GBS group $G$ is large (some finite index subgroup  has a nonabelian free quotient) if and only if $G$ is not a quotient of $BS(m,n)$ with $m,n$ coprime. 
\end{cor}

\begin{proof} 
Groups satisfying the third condition of Proposition \ref{prem} with $m$ and $n$ coprime are those which may be represented by a circle containing no proper plateau, so the corollary follows from Theorem 6.7 of \cite{Le2} (see also \cite{Me}). 
 \end{proof}

\begin{thm} \label{hop}
A 2-generated GBS group $G$ represented by     a reduced labelled graph $\Gamma$ homeomorphic to a lollipop (possibly a circle) maps onto $BS(QX,QY)$ if and only if  $ q_i\wedge r_j =1$ for $0\le i<j<k$, and one of the following holds:
\begin{itemize}
 \item $X$ and $QY$ are coprime;
 \item $Y$ and $QX$ are coprime;
 \item $Q$ and $r_k$ are coprime, and there exists $i_0\in\{0,\dots,\ell-1\}$ such that, if  a prime $p$ divides both $QX$ and $QY$, then it   divides no   $x_i$
with $i>  i_0$ and no  $y_j$ with $j\le i_0$.
\end{itemize}
\end{thm}
 
 Note that $ q_i\wedge r_j =1$ is required only for $j<k$; in particular, there is no condition when $k\le1$. Also note that $i_0$ always exists when $\ell=1$.
 
 We derive two corollaries before giving the proof. 
 
\begin{cor} \label{epi}
A GBS group is epi-equivalent to a Baumslag-Solitar group if and only if it has rank 2 and may be represented by a graph $\Gamma$ satisfying the conditions of Theorem \ref{hop}.
\end{cor}

\begin{proof}
 Suppose that $G$ is  epi-equivalent to $BS(m,n)$ and is 
represented by a reduced labelled graph $\Gamma$. Then $G$ has rank 2, and $\Gamma$ is a lollipop by Lemma \ref{qbs} (it cannot be a segment). Moreover $BS(m,n)$ is isomorphic to $BS(QX,QY)$ by Theorem \ref{deBS}, and $\Gamma$ satisfies the conditions of Theorem \ref{hop}. Conversely, $G$ is epi-equivalent to $BS(QX,QY)$ if $\Gamma$ is as in the theorem.
\end{proof}

\begin{cor} \label{iee}
 Suppose that $BS(m,n)$ is not Hopfian. There exist infinitely many pairwise non-isomorphic GBS groups epi-equivalent to $BS(m,n)$ if and only if $ | m | $ is not prime, $ | n | $ is not prime, and $  m\wedge n \ne1$.
\end{cor}

\begin{proof}
 If $m$ or $n$ is prime, or if $m,n$ are coprime, $BS(m,n)$ has finitely many GBS quotients up to isomorphism by Proposition \ref{infq}. 
Otherwise, non-Hopficity implies that  $m$ and $n$ do not have the same prime divisors \cite{CL}.  We may assume that there is a prime $p$ dividing $n$ but not $m$.  Since   $m\ne\pm n$, and $ | m | , | n | $ cannot be powers of  the same  prime,   $BS(m,n)$ has  infinitely many quotients $G_N$ or $H_N$ constructed in the proof of Proposition \ref{infq}. By the third case of Theorem \ref{hop} ($Q$ and $r_k$ coprime),  the groups $G_N$ are epi-equivalent to $BS(m,n)$ if we construct them using     $p$  as above (not dividing $m$). The  groups $H_N$ are epi-equivalent to $BS(m,n)$  if we choose $\delta=p$, since   $\beta$ and $\delta$ are then coprime.
\end{proof}
 
The remainder of this section is devoted to the proof of Theorem \ref{hop}.

\subsection{Contraction and displacement moves}

We first  describe two general ways of constructing   GBS quotients $G'$ of a GBS group $G$ (see Figures \ref{con} and \ref{dis}). Let $\Gamma$ be a labelled graph representing $G$. 

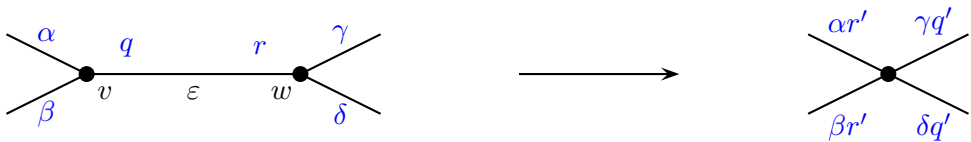
\begin{figure}[h]
\begin{center}
\begin{pspicture}(100,105) 
\rput   (0,50){
\rput(-100,0)
{
   \psline  (0,0)(80,0)   
    \psline  (0,0)(-30,15)
     \psline  (0,0)(-30,-15)
  \psline(80,0)(110,15)
    \psline(80,0)(110,-15)

 \pscircle*( 0,0) {3}
   \pscircle*( 80,0) {3} 
      
   \rput(7,-7) {$v$}
    \rput(73,-7) {$w$}
     \rput(40,-7) {$\varepsilon$}
      {\blue
   \rput(-15,15) {$ \alpha$}
     \rput(-15,-15) {$ \beta$}
 
   \rput(95,15) {$ \gamma$}
     \rput(95,-15) {$\delta$}
 
    \rput(15,10) {$q$}
      \rput(65,10) {$ r$}
             
   }      

 }

\rput(200,0)
{

 \psline  (0,0)(-30,15)
 \psline  (0,0)(-30,-15)
 \psline  (0,0)(30,15)
 \psline  (0,0)(30,-15)

 \pscircle*( 0,0) {3}
 
    {\blue

    \rput(-15,20) {$ \alpha r'$}
     \rput(-15,-20) {$ \beta r'$}
 
   \rput(17,20) {$  \gamma q'$}
     \rput(17,-20) {$ \delta q'$}

   }      

}   
        
   \rput (60,0){ 
\psline[arrowsize=5pt]  {->}(0,0)(60,0)
} 
 }
\end{pspicture}
\end{center}
\caption{contraction move} \label{con}
\end{figure}

\begin{dfn} [Contraction move \cite{DRT}] \label{hoc} Consider an edge $\varepsilon= vw$ of $\Gamma$ with labels $q$ and $r$ (see Figure \ref{con}). The group associated to $\varepsilon$   is  $H=\langle a,b\mid a^q=b^r\rangle$. We may map it onto $\Z$ by sending $a$ to $\frac r{q\wedge r}$ and $b$ to $\frac q{q\wedge r}$ (this amounts to making $H$ abelian, and dividing by torsion if $q\wedge r>1$). Being injective on $\langle a \rangle$ and on $\langle  b\rangle$, this map extends to an epimorphism  
from $G$ to  a GBS group $G'$. In the process the edge $\varepsilon$  gets contracted to a point;   labels near $v$ get multiplied by $r'=\frac r{q\wedge r}$, labels near $w$ by $q'=\frac q{q\wedge r}$. 
\end{dfn}

  \begin{rem}\label{nh} When $q$ or $r$ equals $\pm1$, this is an elementary collapse \cite{Fodef} which does not change $G$ (see Figure \ref{col}). 
  But if none of $q,r$ equals $\pm1$, the epimorphism   $G\twoheadrightarrow  G'$ produced by a contraction move is not an isomorphism because $H$ is not cyclic. 
  This   gives a simple way of showing     the non-Hopficity of Baumslag-Solitar groups (see Figure \ref{hopf}). If $BS(m,n)$ is not Hopfian, then (up to swapping $m$ and  $n$) $n\ne\pm1$ and 
we can write  $m=pm'$ with $p$ a prime not dividing $n$. Collapsing the edge $e$ of Figure \ref{hopf} induces a group isomorphism $\varphi_1$, while contracting $e'$ induces a non-injective epimorphism $\varphi_2$.
\end{rem}

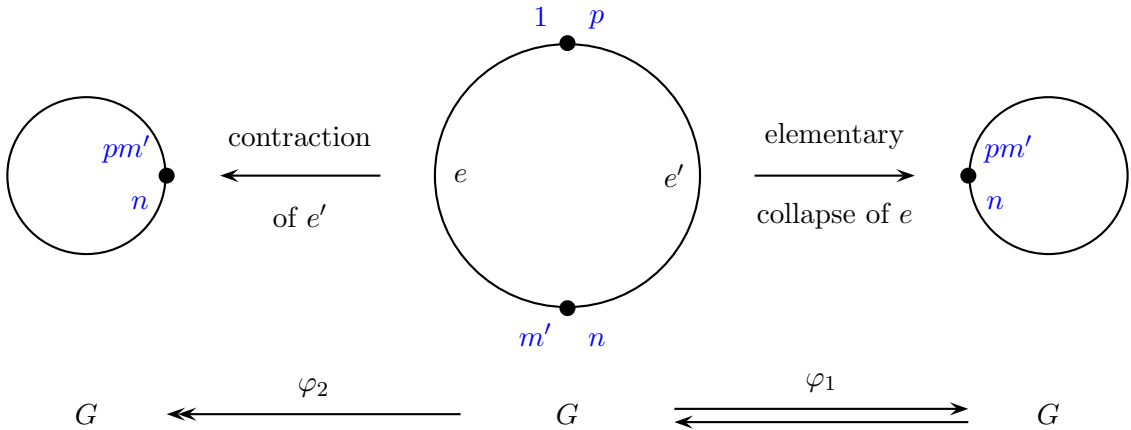
\begin{figure}[h]
\begin{center}
\begin{pspicture}(100,200)
 
\rput  (0,100)   {
 \pscircle ( 50,0) {50} 
  \pscircle (-130,0) {30} 
   \pscircle ( 230,0) {30} 
   
\rput  (0, -90)   {
  \rput ( 50,0) {$G$} 
  \rput (-130,0) {$G$} 
   \rput ( 230,0) {$G$} 
   
\rput(0,2)
{   \psline[arrowsize=5pt]  {->}(90,0)(200,0)
   \psline[arrowsize=5pt]  {<-}(90,-5)(200,-5)  
  \rput(145,10){$\varphi_1$}
   }
   
    \psline[arrowsize=5pt]  {->>}(10,0)(-100,0)
       \rput(-45,10){$\varphi_2$}
}
   
   \rput(-50,15){contraction}
\rput(-50,-15){of $e'$}
   \rput (-80,0){ 
\psline[arrowsize=5pt]  {->}(60,0)(0,0)
} 

   \rput(150,15){elementary}
\rput(150,-15){collapse of $e$}
   \rput (120,0){ 
\psline[arrowsize=5pt]  {->}(0,0)(60,0)
} 

 \pscircle*( 50,50) {3}
  \pscircle*( 50,-50) {3}
 \pscircle*( -100,0) {3}

   \pscircle*( 200,0) {3}
   
  \rput(10,0){$e$} 
    \rput(90,0){$e'$} 
    
    {\blue
   \rput(40,60) {$ 1$}
    \rput(61,59) {$ p$}
  \rput(38,-60) {$ m'$}
    \rput(61,-62) {$ n$}
 
    \rput(215,10) {$ pm'$}
        \rput(210,-10) {$ n$}
      \rput(-115,10) {$ pm'$}
        \rput(-110,-10) {$n$}
      
   }      
           Ê}
        \end{pspicture}
\end{center}
\caption{non-Hopficity of $BS(pm',n)$ when $p,n$ are $>1$ and coprime; $\varphi_1$ is an isomorphism, $\varphi_2$ is not }
\label{hopf}
\end{figure}

\begin{dfn} [Displacement move]\label{hod}
Consider an edge  $\varepsilon=vw$ with labels $q$ and $rs$ (see Figure \ref{dis}). We   replace it by two edges $\varepsilon'$ and $\varepsilon''$, with  labels $q,r$ and $1,s$ respectively (this  is the reverse of an  elementary collapse, it  does not change $G$). We then contract the edge $\varepsilon'$.   If $q$ and $r$ are coprime, the  construction simply ``moves'' $r$:  the label  of $\varepsilon$ near $w$ is divided by $r$, and   labels of edges incident to $v$ are multiplied by $r$. 
\end{dfn}

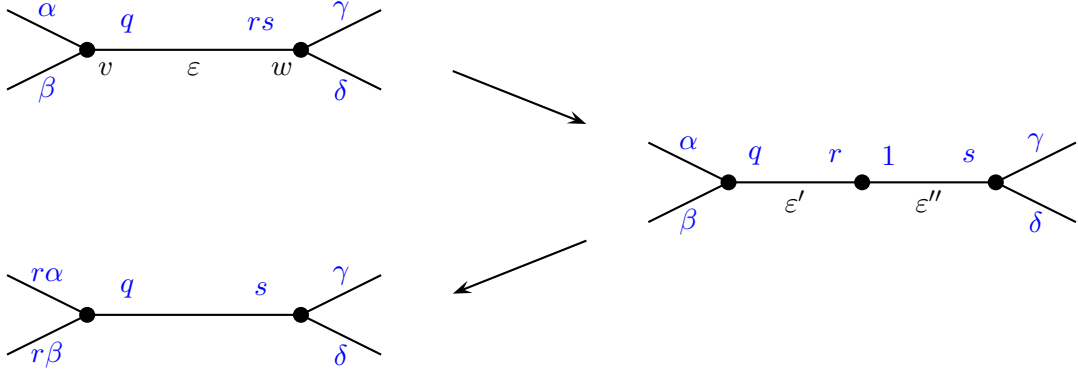
\begin{figure}[h]
\begin{center}
\begin{pspicture}(100,155) 
\rput   (0,120){
\rput(-120,0)
{
   \psline  (0,0)(80,0)   
    \psline  (0,0)(-30,15)
     \psline  (0,0)(-30,-15)
  \psline(80,0)(110,15)
    \psline(80,0)(110,-15)

 \pscircle*( 0,0) {3}
   \pscircle*( 80,0) {3} 
     
   \rput(7,-7) {$v$}
    \rput(73,-7) {$w$}
     \rput(40,-7) {$\varepsilon$}
      {\blue
   \rput(-15,15) {$ \alpha$}
     \rput(-15,-15) {$ \beta$}
 
   \rput(95,15) {$ \gamma$}
     \rput(95,-15) {$\delta$}
 
    \rput(15,10) {$q$}
      \rput(65,10) {$ rs$}      
   }      

 }
 
 \rput(-120,-100)
{
   \psline  (0,0)(80,0)   
    \psline  (0,0)(-30,15)
     \psline  (0,0)(-30,-15)
  \psline(80,0)(110,15)
    \psline(80,0)(110,-15)

 \pscircle*( 0,0) {3}
   \pscircle*( 80,0) {3} 
         {\blue
   \rput(-15,15) {$ r\alpha$}
     \rput(-15,-15) {$r \beta$}
 
   \rput(95,15) {$ \gamma$}
     \rput(95,-15) {$\delta$}
 
    \rput(15,10) {$q$}
      \rput(65,10) {$  s$}      
   }      

 }
  
  \rput(120,-50)
{
   \psline  (0,0)(100,0)   
    \psline  (0,0)(-30,15)
     \psline  (0,0)(-30,-15)
  \psline(100,0)(130,15)
    \psline(100,0)(130,-15)

 \pscircle*( 0,0) {3}
   \pscircle*( 100,0) {3} 
    \pscircle*(50,0) {3}
     \rput(25,-7) {$\varepsilon'$}
        \rput(75,-7) {$\varepsilon''$}
      {\blue
   \rput(-15,15) {$  \alpha$}
     \rput(-15,-15) {$  \beta$}
 
   \rput(115,15) {$ \gamma$}
     \rput(115,-15) {$\delta$}
 
    \rput(10,10) {$q$}
      \rput(90,10) {$  s$}   
       \rput(40,10) {$  r$}    
        \rput(60,10) {$  1$}       
   }      

 }
  \rput (60,0){ 
\psline[arrowsize=5pt]  {->}(-45,-8)(5,-28)
} 

   \rput (60,0){ 
\psline[arrowsize=5pt]  {->}(5,-72)(-45,-92)
} 
 }
\end{pspicture}
\end{center}
\caption{displacement move with $q,r$ coprime} \label{dis}
\end{figure}

\subsection{The case of a circle}

We   prove Theorem \ref{hop} when $\Gamma$ is a circle. We rephrase it as follows:

\begin{prop} \label{cer}
A 2-generated GBS group $G$ represented by     a reduced labelled graph $\Gamma$ homeomorphic to a circle maps onto $BS(X,Y)$ if and only if  there exists $i_0\in\{0,\dots,\ell-1\}$ such that, if  a prime $p$ divides both $X$ and $Y$, then it  divides no $x_i$
with $i>  i_0$ and no  $y_j$ with $j\le i_0$.
\end{prop}

As usual, we assume Convention \ref{bc}:  the  vertex $w_0$ belongs to every plateau, so $G$ is generated by $b_0$ and $\tau$. 

\begin{example} The group $G =\langle a,b,t\mid a^{2\beta}=b^2, tb^\gamma t\m=a^{2\alpha}\rangle$ represented by $\Gamma$ (see Figure \ref{bil}) is 2-generated if and only if $\gamma$ is odd, so let us assume this. The condition of the proposition holds (with $i_0=0$) if and only if no prime divisor of $X\wedge Y=2\beta\gamma\wedge 4\alpha$ divides $x_1=\gamma$; this is equivalent to $\gamma\wedge \alpha=1$.

Write $\gamma=\gamma_1\gamma_2$ with $\gamma_1\wedge \alpha=1$ and $ | \gamma_1 | $ maximal. The proof of the proposition will show that  $G$ is epi-equivalent to the group $G'$ represented by $\Gamma'$; moreover, $G'$ maps onto $BS(X,Y)=BS(2\beta\gamma,4\alpha)$ if and only if $\gamma_2=\pm1$, or equivalently $\gamma\wedge \alpha=1$.
\end{example}

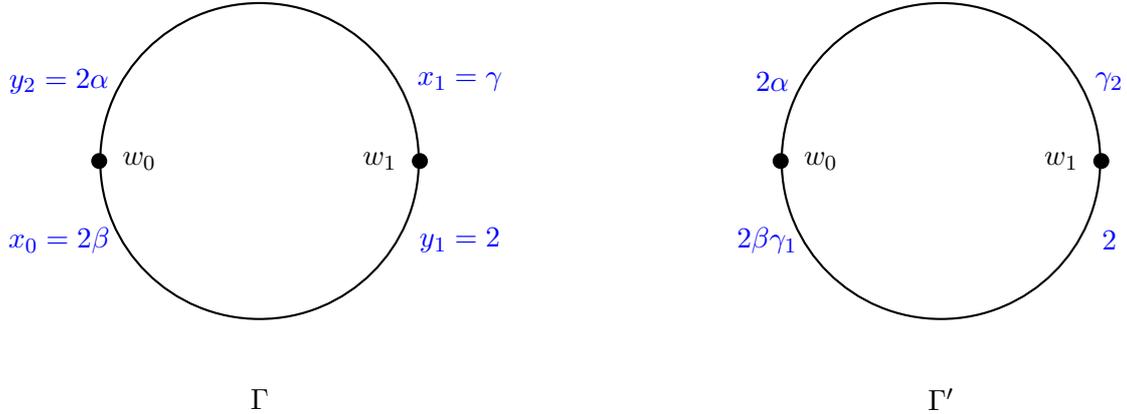
\begin{figure}[h]
\begin{center}
\begin{pspicture}(100,180)

\rput  (15,100)   {
 \pscircle ( 180,0) {60} 

 \pscircle*( 240,0) {3}
 
   \pscircle*( 120,0) {3}
   
   \rput(135,0){$w_0$}
      \rput(225,0){$w_1$}
      
      \rput(180,-90){$\Gamma'$}
      
    {\blue
      \rput(117,30) {$ 2\alpha$}
        \rput(115,-30) {$2\beta\gamma_1$}
        \rput(243,30) {$ \gamma_2$}
          \rput(243,-30) {$ 2 $}
   }        
           Ê}
           
           \rput  (-240,100)   {

 \pscircle ( 180,0) {60} 

 \pscircle*( 240,0) {3}
 
   \pscircle*( 120,0) {3}
   
   \rput(135,0){$w_0$}
      \rput(225,0){$w_1$}
     
          \rput(180,-90){$\Gamma$}
           
    {\blue
      \rput(105,30) {$ y_2=2\alpha$}
        \rput(105,-30) {$x_0=2\beta$}
        \rput(255,30) {$ x_1=\gamma$}
          \rput(255,-30) {$y_1= 2 $}
   }        
           Ê}
        \end{pspicture}
\end{center}
\caption{the group $G$ represented by $\Gamma$  is epi-equivalent to a Baumslag-Solitar group if and only if $\gamma$ is odd and $\gamma\wedge \alpha=1$}
\label{bil}
\end{figure}

\begin{proof} [Proof of Proposition \ref{cer}] 

We may assume that $\Gamma$ has more than one edge (i.e.\ $\ell>1$). 
Among the prime   divisors of $XY$, those which divide both $X$ and $Y$ will be  called bilateral, the others unilateral. 
 
 If $p$ is  a unilateral prime dividing some $x_i$ with $i>0$, a sequence of displacement moves  allows us to move $p$ around the circle from $w_i$ to $w_\ell=w_0$ (through $w_{i+1},\dots,w_{\ell-1}$), dividing $x_i$ by $p$ and multiplying   $x_0$  by $p$ (and leaving all other labels unchanged). 
The group represented by this new  labelled graph (which may fail to be reduced)  is  epi-equivalent to $G$ because another sequence of displacement  moves takes  $p$ back to $w_i$ (through $w_{ 1},\dots,w_{i-1}$). 

We   iterate this construction, obtaining a labelled graph $\Gamma'$  such that no unilateral prime divides $x_i$ for $i>0$, and similarly  no unilateral prime divides $y_j$ for $j<\ell$ (for simplicity we     write $x_i$, $y_j$ rather than $x'_i$, $y'_j$ for the labels carried by $\Gamma'$). The  group $G'$ represented by $\Gamma'$ is epi-equivalent to $G$.

We then  make the graph reduced by collapsing edges with at least one label equal to $\pm1$ (as on Figure \ref{col}).
 We obtain a reduced  labelled graph $\Gamma''$  representing the same group   $G'$. Note that $\Gamma',\Gamma''$ are circles, and  $X,Y$ do not change  during the whole process. 

If there exists $i_0$ as in the proposition, all labels $x_i$
for  $i>  i_0$ and  $y_j$ for $j\le i_0$ in the graph  $\Gamma'$ are equal to $\pm1$: they cannot be divisible by a unilateral prime because of the way $\Gamma'$ was constructed, or by  a bilateral prime because of the  assumption on $i_0$. It follows that all edges of $\Gamma'$, except that between $w_{i_0}$ and $w_{i_0+1}$, get collapsed in $\Gamma''$, so  $G'=BS(X,Y)$  and therefore 
$G$  is epi-equivalent to $ BS(X,Y)$.

Conversely, we assume that $G$ (hence $G'$) maps onto $BS(X,Y)$. 
Suppose for a moment that $\Gamma''$ consists of a single   edge. Let $w_{i_0}$ and $ w_{i_0+1}$ be the endpoints of the unique edge of $\Gamma'$ which does not get collapsed in $\Gamma''$. Since the labels $x_0$ and $y_\ell$ in $\Gamma'$ are not equal to $\pm1$, because $\Gamma$ has more than one edge, collapsibility of the other edges implies  $x_i=\pm1$
for  $i>  i_0$ and  $y_j=\pm1$ for $j\le i_0$ in   $\Gamma'$. This means that the corresponding labels of $\Gamma$ cannot be divisible by a bilateral prime, so $i_0$ satisfies the conditions of the proposition.

We complete the proof of the proposition by showing  that \emph{$G'$ does not map onto $BS(X,Y)$ if $\Gamma''$ has more than one edge}. From now on, all labels $x_i$, $y_j$ are those of $\Gamma''$, and $\ell>1$ is the length of $\Gamma''$. Recall that $\Gamma''$ is reduced, and no unilateral prime divides $x_i$ for $i>0$ or $y_j$ for $j<\ell$. We assume that $f:G'\twoheadrightarrow  BS(X,Y)$ is an epimorphism, and we argue towards a contradiction. 

Let $p$ be a prime dividing $x_1$, and $p'$ a prime dividing $y_1$. They are distinct because $w_0$ meets every plateau of $\Gamma$, and they are bilateral (they each divide both $X$ and $Y$).  Note that $px_0$ divides $X$, and $p'y_\ell$ divides $Y$.

Consider the generator $b_0$ of the vertex group carried by $w_0$.
Write $BS(X,Y)=\langle a,t\mid ta^Xt\m=a^Y\rangle $.
The image of  $b_0$ by $f$  is (conjugate to) a power $a^r$ by Lemma \ref{koc}.  By Lemma \ref{lesp},  applied to the pair  $(f(b_0),f(\tau))$, and Remark \ref{rr}, the number $r$ is coprime with $ X\wedge Y $, hence not divisible by either  $p$ or $p'$.  

The incident edge groups at $w_0$ in $\Gamma''$ are generated by $(b_0)^{x_0}$ and  $(b_0)^{y_\ell}$. We claim that one of them (at least) fixes a unique point in the Bass-Serre tree of  $BS(X,Y)$. 

Otherwise, $rx_0$ and $ry_\ell$ must each be divisible by $X$ or $Y$. Since $rx_0$ cannot be divisible by $X$, because $p  x_0$ divides $X$ but not $rx_0$, it must be divisible by $Y$. Similarly, $X$ divides $ry_\ell$. 
Now $p'  y_\ell$ divides $Y$, hence  $r x_0$, hence $rX$, hence  $r^2 y_\ell$, a contradiction which proves the claim.

It follows that  $f$ factors through the group $\hat G  $ obtained by contracting (in the sense of Definition \ref{hoc}) one of the edges of $\Gamma''$ adjacent to $w_0$, say the edge $w_0w_1$. The labels carried by this edge are $x_0$ and $y_1$, they are not coprime. Indeed, $p'$ divides $y_1$ by definition;   it divides $X$, and it cannot divide $x_i$ for $i>0$ because $w_0$ belongs to every plateau, so it divides $x_0$  (if the contracted edge is $w_0w_\ell$, we use a prime dividing $x_{\ell-1}$ to see that $x_{\ell-1}$ and $y_\ell$ are not coprime). It follows that the numbers $X'',Y''$ associated to $\hat G  $ are smaller than $X$ and $Y$ (they get divided by $x_0\wedge y_1$, see Definition \ref{hoc}), so  $\hat G  $ cannot map onto $BS(X,Y)$ by Lemma \ref{eBS}. Thus $f$ as above cannot exist.
\end{proof}

\begin{cor}
Let $G$ be a 2-generated GBS group   represented by a reduced labelled graph $\Gamma$ homeomorphic to a circle. If there is at most one prime $p$ dividing both $X$ and $Y$, then $G$ is epi-equivalent to $BS(X,Y)$.
\end{cor}

\begin{proof}  The result is clear if $X$ and $Y$ are coprime (see also Proposition \ref{prem}). Otherwise, let $i_0$ be the largest $i$ such that $p$ divides $x_i$. Since $w_0$ belongs to every plateau, $p$ cannot divide $y_j$ for $j\le i_0$, so $i_0$ is as in the
proposition.
\end{proof}

\subsection{The general case }

We can now prove Theorem \ref{hop} in full generality.
It follows from Proposition \ref{cer}  if $\Gamma$ is a circle. 
 
$\bullet$  We next consider the case when $k=\ell=1$ ($\Gamma$ is a lollipop consisting of only  two edges, as in Figure \ref{ptit}). The statement   in this case is that \emph{$G$ maps onto $BS(QX,QY)$ if and only if one of the numbers $ X\wedge QY $, $ Y\wedge QX $, or $ Q\wedge R$ equals 1.}
 
 Write $$G=\langle a_0,b_0,\tau\mid a_0^Q=b_0^R,\,\tau b_0^X\tau\m =b_0^Y\rangle$$ and $$BS(QX,QY)=\langle a,t\mid ta^{QX}t\m=a^{QY}\rangle.$$ 
 
 Every prime divisor of $R$ divides $XY$ (see Lemma \ref{copr}), so we may write $R\tilde R=X^\alpha Y^\beta$ with $\tilde R\in\Z$ and $\alpha,\beta\ge0$. 

Sending $a_0$ to $a^{Y^{\alpha+\beta}}$, $b_0$ to $t^\alpha a^{\tilde R Q} t^{-\alpha}$, and $\tau$ to $t$ defines a homomorphism $f$ from $G$ to 
$BS(QX,QY)$,
because the relations of $G$ are satisfied: $$f(b_0^R)=t^\alpha a^{\tilde R QR} t^{-\alpha}=t^\alpha a^{  Q X^\alpha Y^\beta} t^{-\alpha}
=a^{  Q  Y^{\alpha+\beta}}=f(a_0^Q)
$$
$$
f(\tau b_0^X\tau\m) = t t^\alpha a^{\tilde R QX} t^{-\alpha} t\m
=t^\alpha  a^{\tilde R QY}t^{-\alpha} =
f(b_0^Y).
$$
  The image of $f$ contains $t$ and $a^{Y^{\alpha+\beta}}$, so it contains $a^{QX^{\alpha+\beta}}$.
 If $ Y\wedge QX =1$, the map $f$ is onto, so   $G$ maps onto $BS(QX,QY)$. By symmetry, this is also true   if 
  $ X\wedge QY =1$. 
  
  Now assume $ Y\wedge QX \ne1$ and   $ X\wedge QY \ne1$. 
    If $f:G\twoheadrightarrow BS(QX,QY)$ is an epimorphism,   $f(a_0)$ is (conjugate to)   $a^{r_0}$ with $r_0$ coprime with $Q$ and $ X\wedge Y $ by Lemma \ref{lesp} and Remark \ref{rr} (being 2-generated, $G$ is generated by $a_0$ and $\tau$). In particular, $f(a_0)$ fixes a single point in the Bass-Serre tree of $BS(QX,QY)$ since $Q\ne\pm1$. We claim that $f(a_0^Q)$ also fixes a unique point: otherwise $a^{r_0Q}$ would be a power of $a^{QX}$ or $a^{QY}$, so $r_0$ would be a multiple of $X$ or $Y$; since $r_0$ is coprime with $Q$ and $ X\wedge Y $, this contradicts $ X\wedge QY \ne1$ or $ Y\wedge QX \ne1$.
 
 It follows that any epimorphism from $G$ to $BS(QX,QY)$ factors through the quotient  group $\hat G  $ obtained by contracting the edge $v_0w_0$. The labels carried by that edge are $Q$ and $R$. If they are coprime, then $ \hat G   =BS(QX,QY)$ and $G$ maps onto $BS(QX,QY)$. If $\delta= Q \wedge R >1$, then $\hat G $ is isomorphic to $BS(QX/\delta,QY/\delta)$,   and $G$ does not map  onto $BS(QX,QY)$ since $\hat G $ does not. This proves the theorem when   $k=\ell=1$.
 
$\bullet$   The next case is when $k=1$ and $\ell $ is arbitrary. First assume $ Y\wedge QX =1$ or $ X\wedge QY =1$.   In particular, $X$ and $Y$ are coprime. As in the proof of Proposition \ref{cer}, we may use displacement moves to construct a quotient  $G'$ of   $G$ represented by a lollipop $\Gamma''$ with $k=\ell=1$, without changing $Q$, $ X$ or $Y$: we
  make all labels $x_i$ for $i>0$ and $y_j$ for $j<\ell$ equal to $\pm 1$, and we collapse edges. 
 In this process $x_0$ becomes $  X$ and $y_\ell$ becomes $  Y$ (the new phenomenon is that 
   $R$ gets multiplied by some number dividing $XY$, so we do not claim that $G'$ is epi-equivalent to $G$). Since $G'$ maps onto $BS(QX,QY)$ by the previous case ($\ell=1$), so does $G$. 
   
 If  $ Y\wedge QX \ne1$ and   $ X\wedge QY \ne1$, we argue as in the case  $\ell=1$. Any epimorphism from $G$ to $BS(QX,QY)$ factors through the group $\hat G $ obtained by contracting the edge $v_0w_0$. This group is represented by the graph $\hat\Gamma $ obtained from $\Gamma$ by deleting the edge $v_0w_0$ and multiplying the labels $x_0$ and $y_\ell$ by $Q/\delta$, with $\delta= Q\wedge R $. By Proposition \ref{cer}, it maps onto $BS(QX,QY)$ if and only if $\delta=1$ and there exists $i_0$ as in the statement of the theorem.
 
$\bullet$  Finally, suppose $k>1$. As usual, any epimorphism  $f:G\twoheadrightarrow BS(QX,QY)$ maps $a_0$ to  (a conjugate of)   $a^{r_0}$ with $r_0$ coprime with $Q$ and $ X\wedge Y $. Now $f(a_0^{q_0})$   fixes a unique point in the Bass-Serre tree of $BS(QX,QY)$: otherwise $a^{r_0q_0}$ would be a power of $a^{QX}$ or $a^{QY}$, so $q_1\cdots q_{k-1}$ would divide $r_0$, a contradiction (note that this argument requires $k\ge2$). Thus $f$ factors through the group $\hat G   $ obtained by contracting the edge $v_0v_1$. This contraction deletes the edge and multiplies $q_1$ by $\frac{q_0}{ q_0\wedge r_1} $, so $G$ maps onto $BS(QX,QY)$ if and only if  $q_0,r_1$ are coprime and $\hat G $ does. The theorem now follows by induction on $k$.

\section{Baumslag-Solitar subgroups}\label{bss}

In this section we use immersions of graphs of groups (represented by   maps between labelled graphs)
to find Baumslag-Solitar subgroups in GBS groups. We show that one can characterize the set of moduli of a GBS group $G$ in terms of Baumslag-Solitar subgroups (Proposition \ref{mod}), and we deduce that  $G$  is residually finite if and only if   it is solvable or virtually $F_n\times\Z$. In the last subsection we 
 determine   which Baumslag-Solitar groups  $BS(r,s)$ may be embedded into a given   $BS(m,n)$.
 
 \subsection{Weakly admissible maps}
 
Let $G$ be a GBS group represented by a labelled graph $\Gamma$.  If $\ov G\inc G$ is finitely generated and not free, it acts on the Bass-Serre tree of $\Gamma$ with infinite cyclic stabilizers (see Lemma 2.7 of \cite{FoJSJ}), and this yields a finite labelled graph $\ov \Gamma$ representing $\ov G$, with a map $\pi:\ov \Gamma\to \Gamma$ sending vertex to vertex and edge to edge (we call such a $\pi$ a \emph{morphism}). 

Conversely, one may use certain morphisms to construct subgroups of a given $G$. 
 In \cite{Le2}  
we defined and used ``admissible'' maps in order to represent finite index subgroups of GBS groups. We now weaken the definition  
in order to represent arbitrary subgroups.

    \begin{dfn}[weakly admissible map] \label{wadm}
    Let $\ov\Gamma$ and $\Gamma$ be     labelled graphs.  A  \emph{weakly admissible map} from $\ov\Gamma$ to $\Gamma$   is a pair $(\pi,m)$ where $\pi:\ov\Gamma\to\Gamma$   is a morphism, and $m$  assigns a positive multiplicity   to each vertex and edge of $\ov\Gamma$ so that  the following condition 
    is satisfied (see Figure \ref{eto}):
    given  an edge   $e$  of $\Gamma$, with origin $v$ and label $\lambda_e$ near $v$, and   $x\in\pi\m(v)$, define  $k_{x,e}$ as the gcd   $m_x\wedge\lambda_e$; then there are \emph{at most} $k_{x,e}$ edges of  $\ov\Gamma$ with origin $x$ mapping to $e$, they each have multiplicity $m_x/k_{x,e}$, and their label near $x$ is $\lambda_e/k_{x,e}$. 
    \end{dfn}
           
    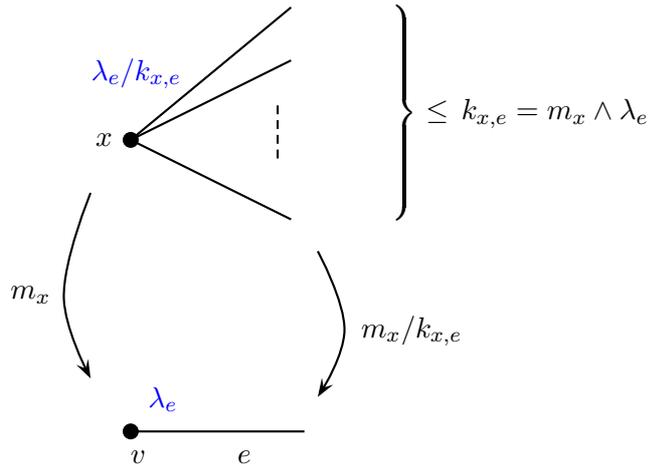
\begin{figure}[h]
\begin{center}
\begin{pspicture}(100,210) 

\rput(0,40){

 \psline  (0,0)(65,0)
     \pscircle*( 0,0) {3}
 \put(0,-12){$v$}
\put(40,-12){$e$}
\rput(12,12){\blue$\lambda_e$}
      \rput (0,110){ 
\rput(2,25){\blue$\lambda_e/k_{x,e}$}        
   \rput(-10,0) {$x$}   
   \rput(152,10){$\le\, k_{x,e}=m_x\wedge\lambda_e$}
   \pscircle*( 0,0) {3}  
       \psline  (0,0)(60,50) 
          \psline  (0,0)(60,30) 
            \psline  (0,0)(60,-30) 
            \psline[linestyle=  dashed, dash=3pt 2pt
            ](55,13)(55,-7)
  \rput*  {90}(100,12){$\underbrace{\phantom{aaaaaaaaaaaaaa}}$    }  
    }
 \rput(-15,80)   
 {
  \pscurve[arrowsize=5pt]  {->}(0,10)  (-10,-30)(0,-60)
    }
   \put(-45,50){$m_x$}   
  \rput(70,68)   
 {
  \pscurve[arrowsize=5pt]  {->}(0,0)  (10,-30)(0,-55)
  \rput(35,-30){$m_x/k_{x,e}$}
    }  
    }
   \end{pspicture}
\end{center}
\setlength\belowcaptionskip{1cm}
\caption{weakly admissible map
} \label{eto}
\end{figure}

    The only difference between admissible (\cite{Le2}, Definition 6.1) and weakly admissible is the insertion of  the words ``at most''. 

We usually denote a weakly admissible map simply by $\pi$, keeping $m$ implicit.
   
   \begin{lem} \label{rwa}
If $\pi:\ov\Gamma\to\Gamma$ is a weakly admissible map between connected labelled graphs, the GBS group $\ov G$  represented by $\ov\Gamma$ embeds into the GBS group $G$ represented by $\Gamma$. 
\end{lem}
  
\begin{proof}
 The quickest way of proving this is to show that $\pi$ defines an immersion of graphs of groups in the sense of \cite{Ba} (it is a covering in the sense of \cite{Ba} if and only if $\pi$ is admissible).  With the notations of    \cite{Ba}, the homomorphisms  $\Phi_a:\cala_a\to \cala'_{\Phi(a)}$ and $\Phi_e:\cala_e\to \cala'_{\Phi(e)}$ are  multiplication by the corresponding multiplicity (all groups are identified with $\Z$).  The commutation relation (2.2) of \cite{Ba} comes from the equality $ 
 \frac{\lambda_e}{k_{x,e}}\cdot m_x=\frac{m_x}{k_{x,e}}\cdot \lambda_e$.  Since $k_{x,e}=m_x\wedge\lambda_e$, multiplication by $m_x$ injects the set of cosets of $\Z$ modulo $\frac{\lambda_e}{k_{x,e}} $ into the set of cosets modulo $\lambda_e$, so $\cala_{a/e}$ embeds into $\cala'_{a'/f}$. The fact that there are at most $k_{x,e}$ lifts ensures that the maps $\displaystyle \Phi_{a/f}:
 \amalg_{e\in\Phi\m_{(a)}(f)}\cala_{a/e}\to \cala'_{a'/f}$ may be made injective.  
 
 Proposition 2.7 of \cite{Ba} asserts that an immersion induces an injection between the fundamental groups of the graphs of groups, so $\ov G$ is a subgroup of $G$. 
\end{proof}
  
  \begin{rem}\label{prac} To prove that   a group $\ov G$ embeds into   $G$, we shall represent $\ov G$ and  $G$ by labelled  graphs $\ov \Gamma$ and $\Gamma$ (not necessarily reduced), define a morphism $\pi:\ov\Gamma\to\Gamma$, and assign multiplicities to vertices 
  of $\ov \Gamma$ (but not to edges).    We then check weak admissibility as follows.  For each oriented edge   $e$ of $\Gamma$, and each $x\in\ov\Gamma$ mapping onto the origin of $e$, we compute  $k=m_x\wedge \lambda_e$. 
  Each oriented edge $\ov e$ of $\ov\Gamma$ with origin $x$ mapping onto $e$ must have label $
    {\lambda_e}/k$,  
  and  there must be at most $k$ such edges (in our constructions, there will be at most 2 edges, so we just need $k>1$); we also have to check that the multiplicity $ 
  {m_x}/k$ of $\ov e$ is the same if we compute it using its terminal point rather than its origin $x$.
    \end{rem}
  
  \subsection{Finding Baumslag-Solitar subgroups}
    
\begin{lem}\label{cont}
If a non-cyclic GBS group $G$ is  a quotient of $BS(m,n)$, with $m$ and $n$ coprime, it contains a subgroup isomorphic to $BS(m,n)$.
\end{lem}

This does not hold if $m$ and $n$ are not coprime:
 $BS(3,5)$ is a quotient de $BS(6,10)$ (see Lemma \ref{eBS}) but does not contain it (see Theorem \ref{obstr}). 
  
\begin{proof}
 By Proposition \ref{prem}, $G$ is represented by a reduced labelled graph $\Gamma$ which is a circle  with $X=m$ and $Y=n$. We use the notations of Subsection \ref{2g} (see Figure \ref{disp}), so  the circle has length $\ell$, with  $X=\prod_{j=0}^{\ell-1}x_j$ and $Y=\prod_{j=1}^{\ell }y_j$. We orient $\Gamma$ counterclockwise, so     edges go from $w_i$ to $w_{i+1}$.
 
 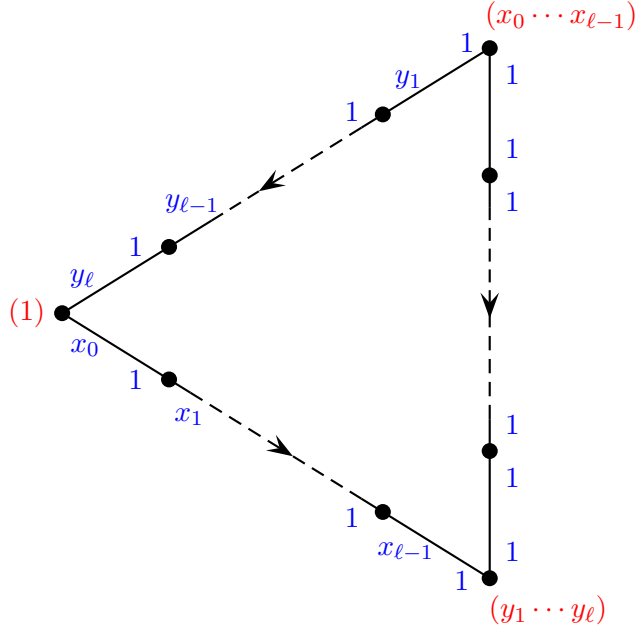
\begin{figure}[h]
\begin{center}
\begin{pspicture}(100,250)

\rput(-40,130){

\psline[linestyle=dashed] (86.4,-54)(112,-70)
  \psline (0,0)(48,-30)
   \psline (160,-100)(112,-70)
  \psline[arrowsize=7pt,linestyle=dashed]  {->}(48,-30)(86.4,-54)

\psline[arrowsize=7pt,linestyle=dashed]  {->}(112,70)(73.6,46)
  \psline (0,0)(56,35)
   \psline (160,100)(112,70)
   \psline[linestyle=dashed] (56,35)(73.6,46)

\psline (160,-100)(160,-40)
\psline (160,100)(160,40)
\psline[arrowsize=7pt,linestyle=dashed]  {->}(160,40)(160,-2)
\psline[linestyle=dashed]   (160,-40)(160,0)

 \pscircle*( 0,0) {3}
  \pscircle*( 160,100) {3}
  \pscircle*( 160,-100) {3}
 \pscircle*( 40,25) {3}
  \pscircle*( 120,75) {3}
   \pscircle*( 40,-25) {3}
  \pscircle*( 120,-75) {3}
 \pscircle*( 160,52) {3}
  \pscircle*( 160,-52) {3}
     
   {\red     
 \rput(-15,0) { $(1)$}    
   \rput(180,-112) { $(y_1\cdots y_\ell)$}  
     \rput(185,112) { $ (x_0\cdots x_{\ell-1})$}  
     }
   {\blue
   \rput(4,13) {$ y_\ell$}
      \rput(5,-13) {$ x_0$}
       \rput(24,25) {$ 1$}
         \rput(24,-25) {$ 1$}
             \rput(45,41) {$ y_{\ell-1}$}
                \rput(44,-39) {$ x_{ 1}$}

 \rput(148,102) {$ 1$}
  
   \rput(126,88) {$ y_1$}
 \rput(105,76) {$ 1$}  
 \rput(146,-101) {$ 1$}
  \rput(125,-90) {$  x_{\ell-1}$}
 \rput(105,-77) {$ 1$}  

 \rput(165,90) {$ 1$}
 \rput(165,62) {$ 1$}
  \rput(165,42) {$ 1$}
   \rput(165,-90) {$ 1$}
   \rput(165,-62) {$ 1$}
    \rput(165,-42) {$ 1$}
   }      

   }
    \end{pspicture}
\end{center}
\caption{a non-reduced graph representing $BS(m,n)=BS(X,Y)$}
\label{trg}
\end{figure}
 
 Let $\ov \Gamma$ be the (non-reduced) labelled graph pictured on Figure \ref{trg}. It is a circle consisting of  $3\ell$ edges,   which we view as 3 blocks of $\ell$ edges each. Going round the circle in the counterclockwise direction starting at the vertex on the left, the first $\ell$ edges have labels $x_0,\dots,x_{\ell-1}$  at their origin and 1 at their terminal point. The next $\ell$ edges have both labels equal to 1. The last $\ell$ edges  have label 1 at their origin and $y_1,\dots,y_\ell$  at their terminal point. Elementary collapses reduce $\ov \Gamma$ to the graph with one vertex and one edge, and   labels $\prod_{j=0}^{\ell-1}x_j$ and $\prod_{j=1}^{\ell }y_j$, so the associated group $\ov G$ is isomorphic to $BS(m,n)$.  
 
 We prove the lemma by describing   a weakly admissible map $\pi:\ov \Gamma  \to  \Gamma$. As indicated on Figure \ref{trg}  by arrows (lifting the orientation of $\Gamma$), the first and third blocks of $\ell$ edges wrap around $\Gamma$ in the counterclockwise direction, the middle block in the clockwise direction. The multiplicities $m_x$ at vertices of $\ov\Gamma$ are as follows (for clarity, we write them within parentheses; three of them are indicated in red on Figure \ref{trg}). In the first block:   $$(1),(y_1), (y_1y_2),\dots,(y_1\cdots y_\ell).$$ In the second block: $$(y_1\cdots y_\ell), (y_1\cdots y_{\ell-1}x_{\ell-1}),\dots,  (y_1\cdots y_{i}x_i \cdots x_{\ell-1}),\dots, (y_1x_1\cdots   x_{\ell-1}),(x_0\cdots x_{\ell-1}).$$  In the third block:
$$ (x_0\cdots x_{\ell-1}),  (x_1\cdots x_{\ell-1}),\dots, (x_{\ell-1}),(1).$$  Weak admissibility is checked as explained in Remark \ref{prac}, using the assumption $\prod x_j\wedge \prod y_j=1$ in the first and third blocks.
\end{proof}

\begin{prop}\label{mod}
 Let $r=\frac mn$ be a non-zero rational number written in lowest terms, with $r\ne\pm1$. Let $G$ be a non-elementary GBS group. Then $r$ belongs to the image of $\Delta_G$ if  and only if $G$ contains a subgroup $H$ isomorphic to $BS(m,n)$.
\end{prop}

See Proposition \ref{el} for the special case $r=\pm1$.

\begin{proof}   
The ``if'' direction is clear since a relation $ta^mt\m=a^n$ with $m\ne\pm n$ implies that $a$ is elliptic, so $r=\frac mn$ is a modulus. Conversely,  if $\frac mn=\Delta_G(t)$ and $a$ is a non-trivial elliptic element, there is a relation $ta^{qm}t\m=a^{qn}$. Replacing $a$ by  a power, we may assume $q=1$.  The group $H=\langle a,t\rangle$  is not free, so is a non-cyclic GBS group (it acts on the Bass-Serre tree of $G$ with infinite cyclic stabilizers, see  Lemma 2.7 of \cite{FoJSJ}). Being a quotient of $BS(m,n)$, the group $H$ 
contains $BS(m,n)$ by Lemma \ref{cont}.
\end{proof}

When the modulus $r$ is an integer, the proposition provides a solvable subgroup  $BS(1,r)$ (see also Lemma 2.4 of \cite{Le}). If $G$ itself is not solvable, we can construct a more complicated  subgroup.

\begin{lem}\label{mode}
 If a non-solvable GBS group $G$  contains $BS(1,n)$ with $n\ne\pm1$, then it contains $BS(q,qn)$ for some prime $q$.
\end{lem}

\begin{proof} 
Consider the action of $G$ on the Bass-Serre tree $T$ associated to a minimal labelled graph $\Gamma$. The assumption  implies that $n$ is a modulus, so let $t$ be a hyperbolic element with modulus $n$, and let $L\inc T$ be its axis. Since $G$ is not solvable, its action on $T$ is irreducible. It is minimal, and there is an elliptic element whose fixed point set is not the whole of $T$. By a general fact on groups acting on trees (see below), there exists an elliptic element $a$ with fixed point set disjoint from $L$. It satisfies a relation $ta^qt\m=a^{qn}$ with $q>1$, but $ta t\m\ne  a^{ n}$.   Replacing $a$ by  a power, we may ensure that these properties hold with $q$  prime. The group  $H=\langle a,t\rangle$ is not solvable, so 
by Proposition \ref{quis} 
it is  isomorphic to $BS(q,qn)$. 

(To prove the fact mentioned above, use irreducibility to find $g\in G$ with $gL\cap L=\es$. If the result is wrong, every elliptic element fixes the bridge between $L$ and $gL$. By minimality, the set of points fixed by all elliptic elements is $T$.)
\end{proof}

 \begin{cor}\label{rf}
A GBS group is residually finite if and only if it is solvable or unimodular (virtually $F\times \Z$ with $F$ free).
\end{cor}

This is proved in \cite{Mes} for $G=BS(m,n)$; in this case $G$ is residually finite if and only if $m=\pm1$, or $n=\pm1$, or $m=\pm n$. It follows from \cite{Wh} that residual finiteness is invariant under quasi-isometry among GBS groups. 

\begin{proof}
A solvable GBS group is isomorphic to some $BS(1,n)$, a unimodular one is virtually $F \times \Z$, these groups are residually finite. 
Conversely, if $G$ is not unimodular and not solvable, it contains a non residually finite Baumslag-Solitar group  by Proposition \ref{mod} if there is a modulus $r$ such that neither $r$ nor $\frac1r$ is an integer, by Lemma \ref{mode} otherwise.
\end{proof}

\subsection{Embeddings of Baumslag-Solitar groups} \label{ebd}

\begin{thm} \label{obstr}
Assume that $BS(r,s)$ is non-elementary, i.e.\  $(r,s)\ne(\pm1,\pm1)$. 
Then  $BS(r,s)$ embeds into $BS(m,n)$ if and only if the following hold:  
\begin{enumerate}
\item $\frac rs$ is a power of $\frac mn$;
\item if $m$ and $n$ are divisible by $p^\alpha$ but not by $p^{\alpha+1}$, with $p$ prime and $\alpha\ge0$, then neither $r$ nor $s$ is divisible by $p^{\alpha+1}$; in particular (for $\alpha=0$), any prime dividing $rs$ divides $mn$;
\item if $m$ or $n$ equals $\pm1$, so does $r$ or $s$.
\end{enumerate}
\end{thm}

See Proposition \ref{el} for the elementary case (when $BS(r,s)$ is equal to $\Z^2$ or $K$).

\begin{example}  \label{mp}
Applied with $p=2$ and $\alpha=1$, the second condition implies  that $BS(12,20)$ does not embed into $BS(6,10)$. It also implies that $BS(2,2)$ only  contains  $BS(1,\pm1)$. On the other hand, $BS(2,3)$ contains   $BS(2^{a+b}3^c,2^{b}3^{a+c})$ for all $a,b,c\ge0$.
\end{example}

\begin{proof}
We first show that the conditions are necessary.

1 follows from Lemma \ref{modeg}: the image of $\Delta_{BS(r,s)}$ in $\Q^*$, generated by $\frac rs$, is contained in the image of $\Delta_{BS(m,n)}$, generated by $\frac mn$.

For 2, we first show that the following property holds in $BS(m,n)$: \emph{if $a$ is elliptic and $g$ is arbitrary, there is a relation $a^N=ga^Mg\m$ with  $N$ not divisible by $p^{\alpha+1}$}. We may assume that $a$ generates the stabilizer of a vertex $v$ in the Bass-Serre tree of $BS(m,n)$. Then $gag\m$ generates the stabilizer of $w=gv$. We apply Lemma \ref{wt} to the segment $vw$. All numbers $q_j$, $r_i$ (for $i>0$) are equal to $m$ or $n$. By the fourth assertion of the lemma, there is a relation $a^N=(gag\m)^M$ with $N$ not divisible by $p^{\alpha+1}$. 

By Lemma \ref{koc}, the property also holds in $BS(r,s)$. It implies that $r$ and $s$ are not divisible by $p^{\alpha+1}$.  

3 is just the fact that a subgroup of a solvable group is solvable.

We shall now construct an embedding  of $BS(r,s)$   into $BS(m,n)$, assuming that the three conditions are satisfied.  
 Writing $BS(x,y)=\langle a_{x,y},t_{x,y}\mid t_{x,y}(a_{x,y})^x t_{x,y}\m=( a_{x,y})^y\rangle$, we say that $q$ is an \emph{index} of an embedding $BS(r,s)\hookrightarrow BS(m,n)$ if $a_{r,s}$ maps to a conjugate of $(a_{m,n})^q$ (a given embedding may have several indices).

\begin{lem} \label{aug}
Let $\nu$ be a non-zero integer. 
\begin{enumerate}
\item $BS(m,n)$ embeds into $BS(\nu m,\nu n)$, with index $\nu$.
\item If there is an embedding   $i:BS(r,s)\hookrightarrow BS(m,n)$  with an  index $q$  
 coprime with $\nu$, then
$BS(\nu r,\nu s)$ embeds into $BS(\nu m,\nu n)$, with an index dividing $q$.
\end{enumerate}
\end{lem}

\begin{proof}
The presentation $BS(\nu m,\nu n)=\langle a,b,t\mid a^ \nu  =b, tb^mt\m=b^n\rangle$ expresses $BS(\nu m,\nu n)$ as the amalgam of $\langle a \rangle$ with    $\langle b,t\rangle\simeq BS(m,n)$ over $\langle a^ \nu \rangle=\langle b\rangle$. This proves 1.

For 2, consider $i:BS(r,s)\hookrightarrow 
  BS(m,n)\inc BS(\nu m,\nu n)$ with 
 $BS(m,n)=\langle b,t\rangle$ as above and $i(a_{r,s})=b^q$.
First assume that none of $r,s$ is equal to $\pm1$. Then $a_{r,s}$ has no proper root in $BS(r,s)$, so $i(BS(r,s))\cap \langle b\rangle$ is generated by $b^q$. Also note that $\langle a^q\rangle\cap \langle b\rangle$ is generated by $b^q=(a^q)^\nu$ because $\nu\wedge q=1$. 
It follows that the subgroup of $BS(\nu m,\nu n)$ generated by $i(BS(r,s))$ and $a^q$ is isomorphic to $BS(\nu r,\nu s)$ (we are using the fact that, given an amalgam $H=A*_CB$ and subgroups $A_1\inc A$, $B_1\inc B$ with $A_1\cap C=B_1\cap C$, the subgroup of $H$ generated by $A_1 $ and $B_1$ is isomorphic to $ A_1*_{A_1\cap B_1}B_1$).

   If $BS(r,s)$ is solvable, define $q' $ (dividing $q$) by $i(BS(r,s))\cap \langle b\rangle=\langle b^{q'}\rangle$, and let $a'_{r,s}=i\m(b^{q'})$. We then have $BS(r,s)=\langle a'_{r,s},t_{r,s}\mid t_{r,s}(a'_{r,s})^r t_{r,s}\m=( a'_{r,s})^s\rangle$, and we may argue as in the previous case.
   \end{proof}

We may assume that $m$ and $n$ are different from $ \pm1$ (the theorem is easy otherwise). We represent $BS(m,n)$ by its standard labelled graph $\Gamma_{m,n}$, oriented so that the edge has initial label $m$ and terminal label $n$.
Using Condition 1, we write $\frac rs=(\frac mn)^\beta$; we may assume $\beta\ge0$. We   write $(r,s)= (\gamma r',\gamma s')$ with $r'\wedge s'=1$, and $(m,n)=(\delta m',\delta n')$ with $m'\wedge n'=1$.

We first assume that no prime appears with the same exponent $\alpha>0$ in $m$ and $n$, so every prime dividing $rs$ divides $m'n'$. We distinguish three subcases. 

$\bullet$ Assume that $m\wedge n=1$. Using the first assertion of Lemma \ref{aug} to increase $\gamma$ if needed, we may assume $r=m^{x+\beta}n^y$ and $s=m^xn^{y+\beta}$ with $x,y\ge1$ and $\beta\ge0$.
We represent $BS(r,s)$ by the non-reduced  labelled graph $\ov \Gamma$ pictured on Figure \ref{gros}, oriented in the counterclockwise direction. Going around the circle starting on the left, one encounters   7 blocks:  

(1)\quad $x+\beta$ edges with labels $m,1$
 
(2)\quad   $x+\beta$ edges with both labels 1
  
(3)\quad $y$ edges with labels $n,1$

(4)\quad $x+y+\beta$ edges with both labels 1

(5)\quad $x$ edges with labels $1,m$

(6)\quad   $y+\beta$ edges with both labels 1

(7)\quad  $y+\beta$ edges with labels $1,n$.
 
  \begin{figure}[h]
\begin{center}
\begin{pspicture}(100,280)
\rput(-40,50){
 \psline (200,160)(200,110)
 
 \psline (200,0)(200,50) 

 \psline[arrowsize=7pt,linestyle=dashed]  {->}(200,50)(200,83)
 \psline[linestyle=dashed] (200,83)(200,110)
  
 \psline (200,160)( 165,160)
 
  \psline[arrowsize=7pt,linestyle=dashed]  {->}(105,160)(145,160)
 \psline[linestyle=dashed] (145,160)(165,160)
  \psline (100,160)( 65,160)
  \psline[arrowsize=7pt,linestyle=dashed]  {->}(0,160)(45,160)
 \psline[linestyle=dashed] (45,160)(65,160)
  \psline (0,0)( 35,0)
   \psline[arrowsize=7pt,linestyle=dashed]  {->}(100,0)(55,0)
 \psline[linestyle=dashed] (55,0)(35,0)
  \psline (100,0)( 135,0)
   \psline[arrowsize=7pt,linestyle=dashed]  {->}(200,0)(155,0)
 \psline[linestyle=dashed] (155,0)(135,0)
 \psline (-80,80)(  -45,45) 
     \psline[arrowsize=7pt,linestyle=dashed]  {->}(-45,45)(-20,20)
 \psline[linestyle=dashed] (-20,20)(0,0)
 \psline (-0,160)(  -35,125) 
     \psline[arrowsize=7pt,linestyle=dashed]  {->}(-35,125)(-60,100)
 \psline[linestyle=dashed] (-60,100)(-80,80)
  
   \pscircle*( 0,0) {3}
  \pscircle*( 200,160) {3}
  \pscircle*( 200,0) {3}
 \pscircle*( 0,160) {3}
  \pscircle*(  100,00) {3}
   \pscircle*( 100,160) {3}
    \pscircle* (-80,80) {3}
    
    {\red
   \rput( 12,15) {$(n^{x +\beta})$}
   \rput( 175,145) {$(n^{x+y+\beta})$}
   \rput( 175,15) {$(m^{x+y+\beta})$}
   \rput( 15,145) {$(m^{y+\beta})$}
   \rput(  100,15) {$(m^{x +\beta})$}
   \rput( 100,145) {$(n^{y+\beta})$}
     \rput (-60,80) {(1)}
    }
   \pscircle*(  70,160) {2}
    \pscircle*(  170,160) {2}
    
     \pscircle*(  30, 0) {2}
    \pscircle*(  130, 0) {2}
    
    \pscircle*(  200,120) {2}
 \pscircle*( 200,40) {2}
  \pscircle*( -55,55) {2}
   \pscircle*( -25,135) {2}

   {\blue
   \rput(8,170) {$ 1$}  
     \rput(78,170) {$ 1$}  
     \rput(63,170) {$ 1$}  
        \rput(90,170) {$ 1$}  
     \rput(-12,2) {$ 1$}  
  \rput(8,-10) {$ 1$}  
     \rput(22,-10) {$ 1$}  
     \rput(37,-10) {$ 1$}  
        \rput(90,-10) {$ 1$}  

\rput(100,0){  
\rput(8,169) {$ m$}  
     \rput(78,169) {$ m$}  
     \rput(63,170) {$ 1$}  
        \rput(90,170) {$ 1$}  
        }
        
   \rput(100,-180){  
\rput(8,169) {$ n$}  
     \rput(22,170) {$ 1$}  
     \rput(37,169) {$ n$}  
        \rput(90,170) {$ 1$}  
        }
        
        \rput(-76,68) {$ m$}
\rput(-65,56) {$ 1$}
\rput(-52,42) {$ m$}

     \rput(-12,158) {$ 1$}  
\rput(-76,92) {$ n$}
\rput(-22,147) {$ n$}
\rput(-35,134) {$ 1$}

\rput(208,150){$ 1$}
\rput(208,130){$ 1$}
\rput(208,110){$ 1$}
\rput(208,50){$ 1$}
\rput(208,30){$ 1$}
\rput(208,10){$ 1$}
  }
  
    \psline[arrowsize=4pt,linewidth=.1pt]  {<->}(0,200)(100,200)
    \psline[arrowsize=4pt,linewidth=.1pt]  {<->}(100,200)(200,200)
  \rput(50,210) {$ y+\beta$}  
    \rput(150,210) {$ x$}  
    
       \psline[arrowsize=4pt,linewidth=.1pt]  {<->}(0,-40)(100,-40)
    \psline[arrowsize=4pt,linewidth=.1pt]  {<->}(100,-40)(200,-40)
  \rput(50,-50) {$ x+\beta$}  
    \rput(150,-50) {$ y$}  
    
           \psline[arrowsize=4pt,linewidth=.1pt]  {<->}(240,0)(240,160)
     \rput(270,80) {$x+ y+\beta$}  

          \psline[arrowsize=4pt,linewidth=.1pt]  {<->}(-30,-30)(-110,50)
     \rput(-82,-2) {$x+ \beta$}  

          \psline[arrowsize=4pt,linewidth=.1pt]  {<->}(-30,190)(-110,110)
     \rput(-82,158) {$  y+\beta$}  
   }
 
   \end{pspicture}
\end{center}
\caption{embedding $BS(m^{x+\beta}n^y,m^xn^{y+\beta})$ into $BS(m,n)$ when $m\wedge n=1$}
\label{gros}
\end{figure}
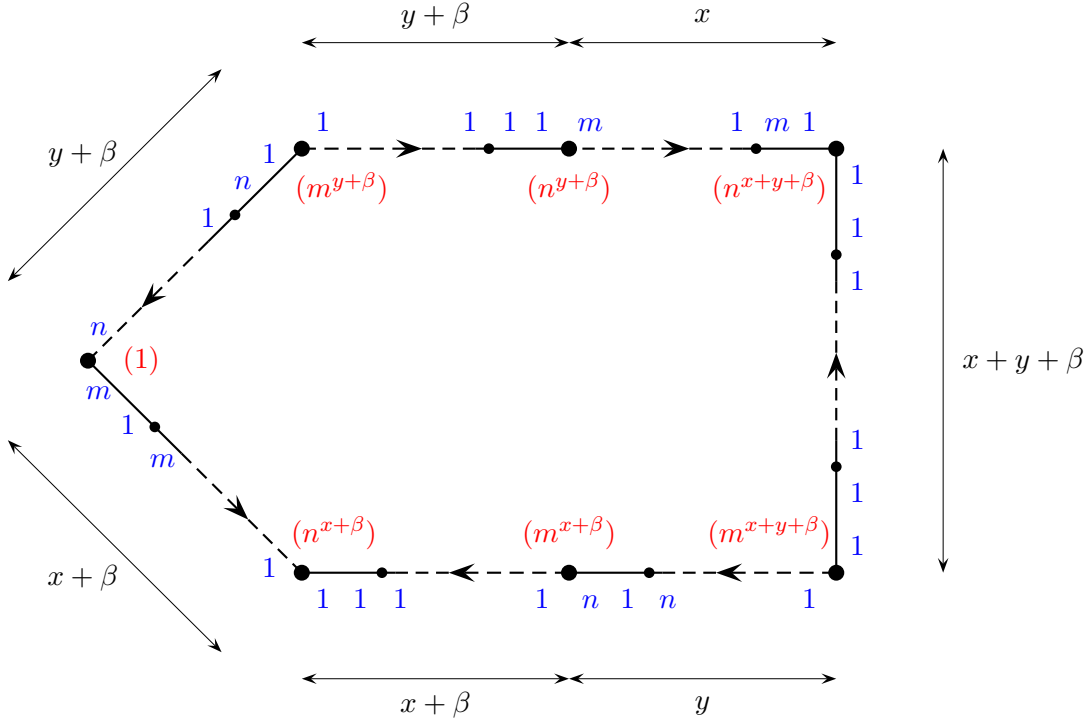

We let $\pi$ be the unique morphism from $\ov \Gamma$ to $\Gamma_{m,n}$ which preserves orientation on blocks $1,4,7$ and reverses it on blocks   $2,3,5,6$ (see  
Figure \ref{gros}, with the  arrows  lifting the orientation of $\Gamma_{m,n}$). 

We define multiplicities of vertices as follows (some of them are indicated within parentheses on Figure \ref{gros}). On each block, they form a geometric progression;  they go from 1 to $n^{x+\beta}$ on block 1, from  $n^{x+\beta}$ to $m^{x+\beta}$ on block 2, from $m^{x+\beta}$ to $m^{x+y+\beta}$ on block 3, 
from $m^{x+y+\beta}$ to $n^{x+y+\beta}$ on block 4,
from $n^{x+y+\beta}$  to $n^{ y+\beta}$ on block 5,
from $n^{ y+\beta}$ to $m^{ y+\beta}$ on block 6,
from $m^{ y+\beta}$ to 1 on block 7. 

One checks that $\pi$ is weakly admissible as explained in Remark \ref{prac}  ($k>1$ is guaranteed by the assumption $m,n\ne\pm1$).

$\bullet$ Assume that $m,n$ are not coprime, but none divides the other. In this case $m'$ and $n'$ are different from $\pm1$, and we can write $BS(r,s)\hookrightarrow BS(m',n')\inc BS(m,n)$ using the previous case and Lemma \ref{aug}.

$\bullet$ If $m $ divides $n$, we   use the fact (proved below) that \emph{$BS(\Delta^x,\Delta^y)$ embeds into $BS(m,\Delta m)$ if $m\ne\pm1$ and $x,y\ge1$} (and $\Delta\ne0$ is arbitrary). We write $n=\Delta m$ and $s=\Delta^\beta r$. Every prime dividing $rs$ divides $\Delta$, so we may assume $r=\Delta^x$ with $x\ge1$. We then have 
$BS(r,s)=BS(\Delta^x,\Delta^{x+\beta})\hookrightarrow BS(m,\Delta m)=BS(m,n)$.

To prove the fact, we represent $BS(\Delta^x,\Delta^y)$ by a labelled graph $\Gamma$ which is a circle consisting of $x+y$ edges  with labels $\Delta$ and 1 placed as on Figure \ref{simp}. We map it to $\Gamma_{m,\Delta m}$ by the unique morphism $\pi$ with the following property:
each edge of $\Gamma$, oriented so that the    label $\Delta$ is near its origin and the label $1$ near its terminal point, maps onto the edge of $\Gamma_{m,\Delta m}$ oriented with $\Delta m$ near its origin and $m$ near its terminal point  (see Figure \ref{simp}).
We assign multiplicity $m$ to all vertices,  multiplicity 1 to  all edges.  One easily checks that $\pi$ is weakly admissible. For future reference, we note that we may also embed $BS(m\Delta^x,m\Delta^y)$   into $BS(m,\Delta m)$, by adding to $\Gamma$ an edge with labels $m$ and 1 (the multiplicity at the terminal vertex of the new graph is $\Delta$).
 
  \begin{figure}[h]
\begin{center}
\begin{pspicture}(100,280)
 
\rput(-90,150){

 \pscurve(0,0)(100,50)(200,0)
 \pscurve(0,0)(100,-80)(200,0)

  \pscircle*( 0,0) {3}
  \pscircle*( 200, 0) {3}
  \pscircle*( 100,-80) {2}
 \pscircle*( 44,-50) {2}
  \pscircle*(  156,-50) {2}
  
  \pscircle*( 59,38) {2}
    \pscircle* (141,38) {2}

   {\blue
   \rput(8,17) {$ \Delta$}  
     \rput(70,52) {$ \Delta$}  
     
     \rput(155,40) {$ \Delta$}  
        \rput(5,-20) {$ \Delta$}  
     \rput(50,-65) {$ \Delta$}  
  \rput(112,-87) {$ \Delta$}  
     \rput(170,-45) {$ \Delta$}  
     \rput(30,-50) {$ 1$}  
        \rput(90,-87) {$ 1$}  
    \rput(42,42) {$ 1$}      
     \rput(150,-65) {$ 1$}      
     \rput(130,52) {$1$}        
       \rput(190,20) {$1$}   
            \rput(192,-20) {$1$}         
  }          
            
 \psline[linestyle=dashed] (-60,0)(0,0)      
  \pscircle*( -60, 0) {2}           
       \rput(-50,10) {{\footnotesize $m$}}        
       \rput(-10,10)  {{\footnotesize $1$}}      
       
                   \psline[arrowsize=4pt,linewidth=.1pt]  {<->}(0,90)(200,90)
                   \rput(100,100) {$ y$}  
            \psline[arrowsize=4pt,linewidth=.1pt]  {<->}( 0,-140)(200,-140)
     \rput(100,-130) {$x$}  

 \rput(275,0){
 
 \psline (0,0)(80,0)
  \psline (-00,-100)(80,-100) 
 
   {\blue
   \rput(10,10) {$ \Delta$}  
     \rput(70,10) {$ 1$}  
     
     \rput(12,-89) {$ \Delta m$}  
        \rput(70,-90) {$ m$}  
     \psline[arrowsize=4pt,linewidth=.1pt]  {->}(40,-30)(40,-70)     
       \pscircle*( 0, 0) {2}  
         \pscircle*( 80, 0) {2}  
           \pscircle*(0, -100) {2}  
             \pscircle*(80, -100) {2}  
        
  }          
 }

   }
   \end{pspicture}
\end{center}
\caption{embedding $BS(\Delta^x,\Delta^y)$ (and $BS(m\Delta^x,m\Delta^y)$)  into $BS(m,\Delta m)$}
\label{simp}
\end{figure}
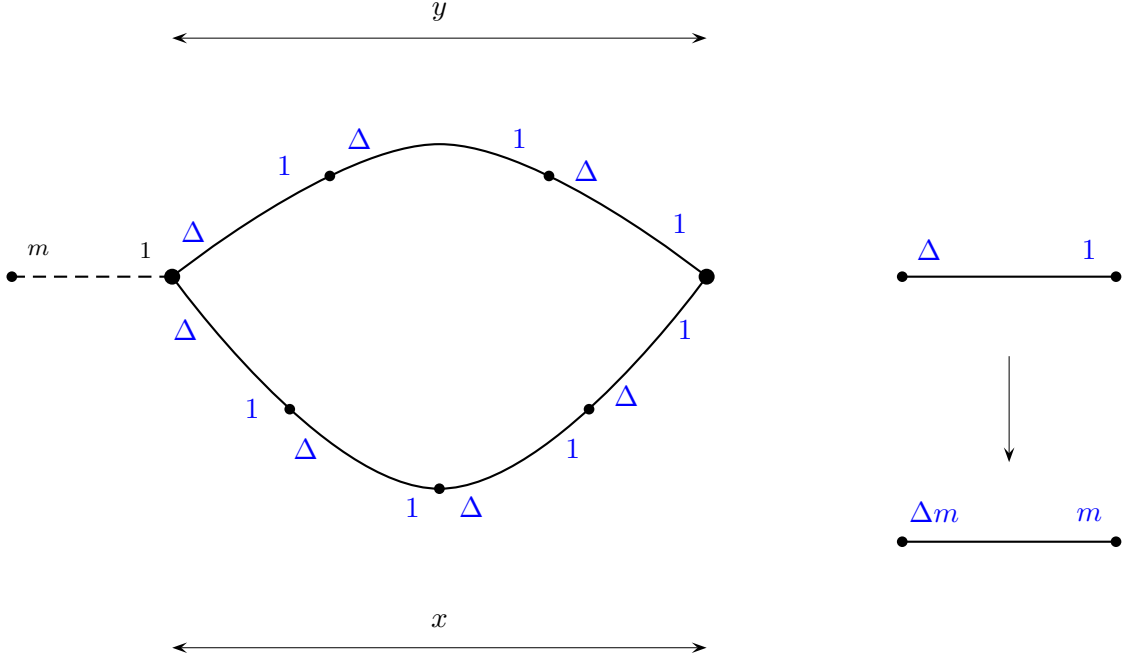

This concludes the case when no prime appears with the same exponent  in $m$ and $n$. 
We observe that, for the weakly admissible maps $\pi$ constructed so far,  multiplicities are divisible only by prime numbers dividing $mn$. With the notations of Lemma \ref{aug}, this implies that we constructed an embedding   $BS(r,s)\hookrightarrow BS(m,n)$  with an index 
$q$ divisible only by primes dividing $mn$.

In the general case, we let $\delta_1$ be the product of the primes appearing with the same exponent in $m$ and $n$, and we
write $(m,n)= (\delta_1m_1,\delta_1n_1)$. Note that 
  $\delta_1\wedge m_1n_1=1$, but $m_1,n_1$ do not have to be coprime. Using Conditions 1 and  2 of the theorem, and Lemma \ref{aug}, we may assume $(r,s)= (\delta_1r_1,\delta_1s_1)$ with $\delta_1\wedge r_1s_1=1$. Every prime dividing $r_1s_1$ divides $m_1n_1$.

  If none of $m_1,n_1$ equals $\pm1$, the previous analysis allows us to embed $BS(r_1,s_1)$ into $BS(m_1,n_1)$ with an index coprime with  $\delta_1$.  
  The second assertion of  Lemma \ref{aug} embeds $BS(r,s)=BS(\delta_1r_1,\delta_1s_1)$ into $BS(\delta_1m_1,\delta_1n_1)=BS(m,n)$.
  
 Now suppose $m_1=1$, so $\delta_1=m\ne\pm1$.  Every prime dividing $r_1s_1$ divides $\Delta=n_1$, so (as above) we may assume   $r_1=\Delta^x$ and $s_1=\Delta^{x+\beta}$ with $x\ge1$. Using a fact proved earlier, we may embed $BS(r,s)=BS(m\Delta^x,m\Delta^{x+\beta})$ into $BS(m,\Delta m)=BS(m,n)$.
\end{proof}

\begin{prop}\label{el}
\begin{itemize}
\item
 $\Z^2$ embeds into every Baumslag-Solitar group, except $BS(1,n)$ for $n\ne\pm1$. The Klein bottle group $K$ embeds into $BS(m,-m)$ and $BS(2m,n)$ for $n\ne\pm1$, and into no other. 
 \item More generally, every GBS group other than $BS(1,n)$ contains $\Z^2$. A GBS group $G$ other than $BS(1,n)$ contains $K$ if and only if $-1$ is a modulus of $G$, or $G$ may be represented by a labelled graph $\Gamma$ with at least one even label.
 \end{itemize}
\end{prop}

\begin{proof}[Sketch of proof]
 In most cases (in particular  if no label of $\Gamma$ equals $\pm1$), one embeds $\Z^2$ or $K$ by observing that $\langle a,b\mid a^m=b^n\rangle$ contains $\Z^2$ if $m$ and $n$ are $\ne\pm1$, and contains $K$ if in addition $m$ or $n$ is even. We leave the remaining cases to the reader. Conversely, if $G$ contains $K= \langle a,t\mid tat\m=a\m\rangle$, then $-1$ is a modulus if $a$ is elliptic. If $a$ is hyperbolic, $t$ acts as a reflection on its axis, and this forces some label to be even.
\end{proof}

\section{Subgroups of $BS(n,n)$}

We denote by $Z(G)$ the center of a GBS group $G$. If $G\ne\Z^2$, it is trivial or infinite cyclic.

 \begin{thm} \label{subnn}
Given a non-cyclic GBS group $G$, and $n\ge2$, the following are equivalent:
\begin{enumerate}
\item $G$ embeds into $BS(n,n)$;
\item $G/Z(G)$ is a free product of cyclic groups, with the order of each finite factor 
 dividing $n$;
\item $G$ may be represented by a reduced labelled graph $\Gamma$ such that all labels near any given vertex $v$ are equal, and they divide $n$.
\end{enumerate}
\end{thm}

This does not hold for $n=1$, because of $K$ and $F\times\Z$ (with $F$ free).

If $\Gamma$ is as in 3, so is any reduced labelled graph $\Gamma'$ representing $G$ (up to admissible sign changes), because $\Gamma$ and $\Gamma'$ differ by slide moves (\cite{FoGBS}, Theorem 7.4).

There is a similar characterization of groups with embed into $BS(n,\pm n)$. One must replace $Z(G)$ by an infinite cyclic normal subgroup, and allow labels near a given vertex to be opposite. In particular, \emph{$G$ embeds into $BS(n,n)$ or $BS(n,-n)$ if and only if it may be represented by a reduced labelled graph $\Gamma$ such that all labels near any given vertex $v$ are equal up to sign, and they divide $n$.}

\begin{proof} We may assume that $G$ is non-elementary. We let $\Gamma_{n,n}$ be the standard splitting of   $BS(n,n)$ as an HNN extension, and we orient its edge. We denote by $T$ the associated Bass-Serre tree. The stabilizer of any edge is equal to the center of $BS(n,n)$.

If $G\inc BS(n,n)$, we consider its action  on the minimal $G$-invariant subtree   $T_{min}\inc T$.  
 Being a non-cyclic GBS group, $G$ is one-ended (\cite{FoJSJ}, Lemma 2.6), so edge stabilizers are non-trivial.  If $g\in G$ fixes an edge, it is central in $BS(n,n)$ hence in $G$. In particular, $Z(G)$ is infinite cyclic. The action of $G$ on $T_{min}$ being irreducible because $G$ is non-elementary, a standard argument implies that $Z(G)$ acts as the identity on $T_{min}$ (see \cite{Le}, Proposition 2.5).

We have shown that every edge stabilizer for the action of $G$ on $T_{min}$ is infinite cyclic and equal to $Z(G)$.  For the induced action of $G/Z(G)$, edge stabilizers are trivial and vertex stabilizers are finite cyclic groups. Their order divides $n$ because they may be expressed as $G\cap\langle a\rangle/G\cap\langle a^n\rangle$ for some $a\in BS(n,n)$. We have  proved that 1 implies 2.

If $G$ is as in 2, we let $G/Z(G)$ act on a reduced tree $S$ with trivial edge stabilizers, and vertex stabilizers cyclic of order dividing $n$. Lift the action to $G$. Edge stabilizers are infinite cyclic (they are non-trivial by one-endedness of $G$). Vertex stabilizers are finite extensions of $\Z$, so are infinite cyclic because they are torsion-free. 

Near each vertex $v$, all edge stabilizers are equal, and their index in the   stabilizer of $v$  divides $n$. This means that in the labelled graph $\Gamma=S/G$ all labels near a given vertex are equal in absolute value. They are actually equal if we fix a generator $z$ of $Z(G)$ and require that all generators of vertex and edge groups of $\Gamma$ be positive roots of $z$. We have proved that 2 implies 3.

Finally, let $G$ be as in 3. Using elementary expansions (the converse of elementary collapses), we represent $G$ by a (non-reduced) labelled graph $\Gamma'$ such that all labels near a non-terminal vertex are equal to 1, and these vertices are trivalent (see an example on Figure \ref{triv}). It is easy to orient edges of $\Gamma'$ so that each trivalent vertex has at least one incoming and one outgoing edge. Let then $\pi$ be the only orientation-preserving graph morphism from $\Gamma'$ to $\Gamma_{n,n}$.
We now define multiplicities so that $\pi$ is weakly admissible in the sense of Definition \ref{wadm}.

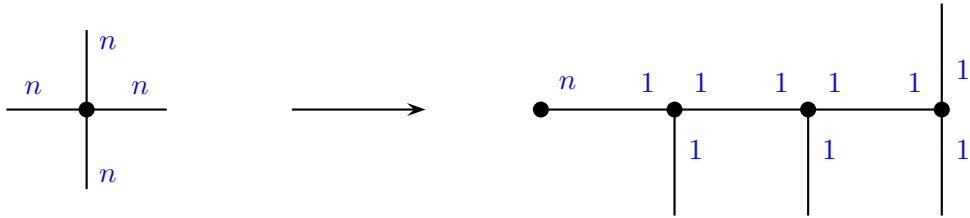
\begin{figure}[h]
\begin{center}
\begin{pspicture}(100,105) 
\rput   (0,50){

\rput(70,0)
{
   \psline  (0,0)(150,0)   
    \psline  (50,0)(50,-40)
     \psline  (100,0)(100,-40)
  \psline(150,40)(150,-40)
 
 \pscircle*( 0,0) {3}
   \pscircle*( 50,0) {3} 
   \pscircle*( 100,0) {3} 
\pscircle*( 150,0) {3} 
   
      {\blue
   
   \rput(58,-15) {$1$}
     \rput(108,-15) {$1$}
      \rput(158,-15) {$1$}  
   \rput(158,15) {$1$}  
    \rput(10,10) {$n$}
      \rput(60,10) {$ 1$}
             \rput(40,10) {$ 1$}
        \rput(90,10) {$ 1$}
             \rput(110,10) {$ 1$}
     \rput(140,10) {$ 1$}
   }      

 }
\rput(-100,0)
{

 \psline  (0,0)(0,30)
 \psline  (0,0)(-30,0)
 \psline  (0,0)(0,-30)
 \psline  (0,0)(30,0)

 \pscircle*( 0,0) {3}
 
    {\blue

    \rput(8,25) {$ n$}
     \rput(20,8) {$ n$}
 
   \rput(-20,8) {$ n$}
     \rput(8,-25) {$n$}
        
   }      

}   
   \rput (-25,0){ 
\psline[arrowsize=5pt]  {->}(0,0)(50,0)
} 
 }
\end{pspicture}
\end{center}
\caption{making non-terminal vertices trivalent with all labels 1} \label{triv}
\end{figure}

All edges have multiplicity 1. A terminal vertex of $\Gamma'$ with label $d$ has multiplicity $n/d$. Trivalent vertices of $\Gamma'$ have multiplicity $m$. The number of incoming (resp.\  outgoing) edges is at most 2, which is not bigger than $k_{x,e} $ since $k_{x,e} =n\wedge n=n$ and we assume $n\ge 2$.

\end{proof}

\begin{cor}\label{mmm}
A finitely generated group $G$ embeds into some $BS(n,n)$ if and only if $G$ is torsion-free, $Z(G)$ is cyclic, and $G/Z(G)$ is a free product of cyclic groups. 
\end{cor}

Theorem \ref{mmmm} is a restatement of this corollary.

\begin{proof}
Suppose that $G$ is as in the corollary. If $ Z(G)$  is trivial, then $G$ is free and embeds into any $BS(n,n)$ with $n\ge2$. If  $ Z(G)$ is infinite cyclic, the proof of $2  \Ra 3$ above shows that $G$ is a GBS group, so it embeds into some $BS(n,n)$ by the theorem. Conversely, any finitely generated subgroup of a GBS group is free or a GBS group, and Theorem \ref{subnn} applies.
\end{proof}

  \begin{flushleft}
\bigskip
Gilbert Levitt\\
Laboratoire de Math\'ematiques Nicolas Oresme\\
Universit\'e de Caen et CNRS (UMR 6139)\\
BP 5186\\
F-14032 Caen Cedex\\
France\\
\emph{e-mail:} \texttt{levitt@unicaen.fr}\\

\end{flushleft}

 \end{document}